\documentclass[12pt,twoside]{amsart}
\usepackage{amssymb}
\usepackage{verbatim}
\usepackage{amsmath}
\usepackage{bm}
\usepackage{a4wide}
\usepackage[latin1]{inputenc}
\usepackage[T1]{fontenc}
\usepackage{times}
\usepackage{amssymb,latexsym}
\usepackage{enumerate}
\usepackage{color}

\newcommand{\C}{\ensuremath{\mathbb{C}}}

\makeatletter
\newcommand{\sumprime}{\if@display\sideset{}{'}\sum%
            \else\sum'\fi}
\makeatother

\begin{document}

\numberwithin{equation}{section}

\newtheorem{theorem}{Theorem}[section]
\newtheorem{proposition}[theorem]{Proposition}
\newtheorem{conjecture}[theorem]{Conjecture}
\def\theconjecture{\unskip}
\newtheorem{corollary}[theorem]{Corollary}
\newtheorem{lemma}[theorem]{Lemma}
\newtheorem{observation}[theorem]{Observation}
\newtheorem{definition}{Definition}
\numberwithin{definition}{section} 
\newtheorem{remark}{Remark}
\def\theremark{\unskip}
\newtheorem{kl}{Key Lemma}
\def\thekl{\unskip}
\newtheorem{question}{Question}
\def\thequestion{\unskip}
\newtheorem{example}{Example}
\def\theexample{\unskip}
\newtheorem{problem}{Problem}

\thanks{Supported by NSF grant 11771089 and Gaofeng grant  from School of Mathematical Sciences, Fudan University}

\address{Department of Mathematical Sciences, Fudan University, Shanghai, 20043, China}

\email{boychen@fudan.edu.cn}

\address{Department of Mathematical Sciences, Fudan University, Shanghai, 20043, China}

\email{ypxiong18@fudan.edu.cn}

\title[Curvature and Bergman spaces]{Curvature and $L^p$ Bergman spaces on complex submanifolds in ${\mathbb C}^N$}

 \author{Bo-Yong Chen and Yuanpu Xiong}
\date{}

\begin{abstract}
 Let $M$ be a closed complex submanifold in ${\mathbb C}^N$ with  the complete K\"ahler metric  induced by the Euclidean metric. Several finiteness   theorems on  the $L^p$ Bergman space  of  holomorphic sections of a given Hermitian line bundle $L$ over $M$ and the associated $L^2$ cohomology groups are obtained. Some infiniteness theorems are also given in order to test the accuracy of  finiteness theorems. As applications we obtain some rigidity results concerning growth of curvatures.
\bigskip

\noindent{{\sc Mathematics Subject Classification} (2010): 32Q15, 32Q28}

\smallskip

\noindent{{\sc Keywords}: $L^p$ Bergman space, $L^2-$cohomology group, curvature, complex submanifold.}
\end{abstract}
\maketitle

\section{Introduction}

 The interest of closed complex submanifolds in $\mathbb C^N$ is twofold. The first is that these complex manifolds are exactly\/ {\it all}\/ Stein manifolds--a central subject in several complex variables; the second is that they are minimal submanifolds in Euclidean spaces which enjoy a rich geometry.

 In what follows we shall denote by $M$ a closed complex submanifold  of dimension $n$ in ${\mathbb C}^N$ and $g$  the restriction  of the Euclidean metric on $M$. A basic question is

 \begin{problem}
 What is the relation between function theory  and  geometry of $M$?
\end{problem}

A beautiful result in this direction is the cerebrated result of Stoll \cite{Stoll} that  $M$ is affine-algebraic if and only if\/ $|M(x,r)|=O(r^{2n})$ for some (and any) $x$, where $M(x,r)$ is the intersection of $M$ with Euclidean ball $B(x,r)\subset \mathbb C^N$, and $|\cdot|$ means the volume with respect to $g$.
It is natural to expect to characterize affine-algebraicity through growth of curvature.  However, Vitter \cite{Vitter} found examples of smooth affine-algebraic varieties that some of them have Ricci curvature decaying faster than quadradically while others  do not. On the other hand, Cornalba-Griffiths \cite{CornalbaGriffiths} characterized the property that a Hermitian vector bundle over a smooth affine-algebraic variety has an algebraic structure through growth of curvature. 

The goal of this paper is to give some finiteness theorems on the $L^p$ Bergman space and $L^2$ cohomology groups associated to a Hermintian line bundle over a closed complex submanifold  in $\mathbb C^N$ through growth of curvature or affine-algebraicity. We remark that  a large literature exists for vanishing theorems and finiteness theorems for holomorphic sections of vector bundles and high cohomology groups on noncompact complete K\"ahler manifolds  (see, e.g., \cite{AndreottiGrauert}, \cite{DemaillyAffine}, \cite{Mok}, \cite{Ni}, or \cite{OhsawaBook}, Chapter 2, \S\,3).

\begin{definition}
For a Hermitian line bundle $(L,h)$ over $M$ we define the $L^p$ Bergman space  $H^0_{L^p} (M,L)$ to be the set of holomorphic sections $s$ of $L$ over $M$ satisfying
   $
   \int_M |s|^p_{h} dV_g<\infty.
   $
   Here $|\cdot|_h$ is the point-wise norm with respect to $h$ and $dV_g$ denotes the volume element of $g$.
\end{definition}

Let $\Theta_h$ be the curvature of $h$ and  ${\rm Tr}_g(\Theta_h)$ the trace of $\Theta_h$ with respect to $g$. Let $\rho$ denote the restriction of $|z|$ on $M$. For the sake of convenience, we always assume that $\rho>0$ on $M$. If an inequality $\lambda(x) \le \mu(x) $ (resp. $\lambda(x) \ge \mu(x) $) holds for $\rho(x)$ large enough, then we shall call it an asymptotic inequality on $M$ and write $\lambda \le_{\rm as} \mu$
   (resp. $\lambda \ge_{\rm as} \mu$).

    \begin{theorem}\label{th:finite_0}
   Suppose $n>1$ and there exists $\alpha>0$ such that
 \begin{equation}\label{eq:trace_1}
  {\rm Tr}_g(\Theta_h) \le_{\rm as}  \alpha/\rho^2.
 \end{equation}
Then  for every $0<p< \frac{2n(n-1)}{\alpha}$ we have
$$
{\rm dim}_{\mathbb C}H^0_{L^p} (M,L)<\infty.
$$
Moreover, if\/ $ {\rm Tr}_g(\Theta_h) \le  \alpha/\rho^2$ holds on whole $M$, then $H^0_{L^p} (M,L)=\{0\}$.
 \end{theorem}

  \begin{remark} 
 \rm (1) In particular, $
{\rm dim}_{\mathbb C}H^0_{L^p} (M,L)<\infty
$ for all $p$ if
 $
  \limsup_{\rho\rightarrow \infty} \rho^2\cdot{\rm Tr}_g(\Theta_h) =0.
 $

 \rm (2) In \S\,7 we shall show that if $L$ is semipositive and satisfies  $ \Theta_h \ge_{{\rm as}}  \alpha\cdot \omega/\rho^2$ with some $\alpha>8n$, where $\omega$ is the K\"ahler form of $g$, then $
{\rm dim}_{\mathbb C}H^0_{L^2} (M,K_M\otimes L)=\infty.
$
 \end{remark}

\begin{example}
 {\rm Let $M=\mathbb C^n\times \{1\}\subset \mathbb C^{n+1}$ and let $L$ be the trivial bundle with metric $h=e^{-\beta \log \rho^2}$.  Note that
$$
  \mathrm{Tr}_g(\Theta_h)=\frac{\beta}{4}\Delta\log\rho^2=\frac{(n-1)\beta}{\rho^2}.
  $$
 Theorem \ref{th:finite_0} yields the trivial fact that if $f$ is an entire  function with $\int_{\mathbb C^n} |f|^2/(1+|z|^2)^{\beta}<\infty$ then  $f\equiv0$ whenever $\beta<n$, which is almost the best possible as   $\int_{\mathbb C^n}1/(1+|z|^2)^{\beta}<\infty$ for  $\beta>n$.} 
\end{example}

  Our proof of Theorem \ref{th:finite_0} is based on a Caccioppoli-type inequality and a Hardy-type inequality. The latter explains why we need the assumption that $n>1$.  The underlying idea is in part motivated by Yau's proof of his $L^p$ Liouville theorem for   subharmonic functions on  complete Riemannian manifolds (cf. \cite{Yau}).   
For holomorphic curves we have

 \begin{theorem}\label{th:finite}
  Suppose $n=1$ and  there exists a continuous increasing function  $\mu:(0,\infty)\rightarrow (0,\infty)$  with
$
\int_1^\infty dt/\mu(t)<\infty
$
such that
 \begin{equation}\label{eq:scalar}
  {\rm Tr}_g(\Theta_h) \le_{{\rm as}}  1/{\mu (\rho^2)}.
 \end{equation}
Then the linear space of holomorphic sections of $L$  satisfying 
$$
\int_M |s|_h^p \, \rho^{-\alpha} dV_g<\infty
$$
 is of finite-dimensional for each $\alpha<2$ and\/ $0<p<\infty$. In particular, $
{\rm dim}_{\mathbb C}H^0_{L^p} (M,L)<\infty$ for each
 $0<p<\infty$.
 \end{theorem}

An interesting case is when $L=K_M^{\otimes m}$ is the pluricanonical bundle and $h=(dV_g)^{-m}$. We call $P_{m,L^p}(M):={\rm dim}_{\mathbb C}H^0_{L^p} (M,K_M^{\otimes m})$ the $m-$th $L^p$\/ {\it plurigenus}\/ of $M$. It is important to realize that $P_{m,L^{2/m}}(M)$ is intrinsic to the complex structure of $M$ (see e.g., \cite{Sakai}).

   \begin{corollary}\label{cor:finite}
   \begin{enumerate}
       \item   Suppose $n>1$ and there exists $\alpha>0$ such that
 the scalar curvature of $g$ satisfies
$
  {\rm Scal}_g\ge_{{\rm as}}  - \alpha/\rho^2.
 $
Then
$
P_{m,L^p}(M)<\infty
$
for all $m\ge 1$ and\/ $0<p<\frac{2n(n-1)}{\alpha m}$.
       \item  Suppose $n=1$ and there exists a continuous increasing function  $\mu:(0,\infty)\rightarrow (0,\infty)$  with
$
\int_1^\infty dt/\mu(t)<\infty
$
such that the Gauss curvature of $g$ satisfies
$
  {\rm Gauss}_g \ge_{{\rm as}}  - 1/{\mu (\rho^2)}.
$
Then $P_{m,L^p}(M)<\infty$ for all\/ $m\ge 1$ and $0<p<\infty$.
       
   \end{enumerate}
   \end{corollary}

 Another interesting case is when $L$ is the trivial line bundle. It follows from Theorem \ref{th:finite} that if $n=1$ and $\varphi$ is a smooth function on $M$ and is subharmonic (w.r.t. $g$) outside a compact subset, then the linear space of holomorphic functions on $M$ satisfying 
$$
\int_M |f|^2 e^\varphi \rho^{-\alpha} dV_g<\infty
$$
 is of finite-dimensional for all $\alpha<2$. It may happen that $\varphi\rightarrow -\infty$ as $\rho\rightarrow \infty$, e.g., when $M$ is parabolic.

\begin{definition}
We define $\mathcal O_d(M)$ to be the linear space of holomorphic functions on $M$ with polynomial growth of order at most $d$. That is, $f\in \mathcal O_d(M)$ if $f$ is holomorphic on $M$ and satisfies $|f|\lesssim \rho^d$.
\end{definition}

 \begin{corollary}\label{cor:finiteness_2}
Suppose $n=1$ and the following conditions hold:
\begin{enumerate}
\item There exists a real-valued smooth function $\varphi$ on $M$ which is  subharmonic outside a compact subset and satisfies $e^\varphi \lesssim \rho^{-\alpha}$ for some $\alpha>0$.
\item There exists $\beta>0$ such that $|M(0,r)|\lesssim r^\beta$ for all $r$.
\end{enumerate}
Then $\mathcal O_d(M)$ is of finite-dimensional for each $d$.
\end{corollary}
\begin{proof}
For each $f\in \mathcal O_d(M)$  we have
\begin{eqnarray*}
\int_{M} |f|^2 e^{m\varphi} dV_g & = & \int_{\rho<1}  |f|^2 e^{m\varphi} dV_g + \sum_{k=1}^\infty \int_{2^{k-1}\le \rho <2^k}  |f|^2 e^{m\varphi} dV_g\\
& \lesssim & \int_{\rho<1}  |f|^2 e^{m\varphi} dV_g + \sum_{k=1}^\infty 2^{k(2d-m\alpha+\beta)}\\
& < & \infty
\end{eqnarray*}
provided $m>\frac{\beta+2d}{\alpha}$. Thus $\dim_{\mathbb C}\mathcal{O}_d(M)<\infty$.
\end{proof}

\begin{example}
{\rm Let $M$ be an  affine-algebraic curve. We may realize $M$ as a branched covering of $\mathbb C$ as in Proposition \ref{prop:affine criteria}. 
Let $\psi$ be a smooth function on $\mathbb C$ such that $\psi(t)=-\log|t|$ outside a bounded neighborhood of $0$. Set
$
\varphi(z)=\psi(z_1).
$
Then $\varphi|_M$ is subharmonic outside a compact set and satisfies $\varphi|_M\le_{\rm as} -\log \rho+C$. Thus ${\rm dim}_{\mathbb C}\, \mathcal O_d(M)<\infty$ for each $d$.}
\end{example}

It is of classical interest to consider the relations between the intrinsic geometry associated to invariant metrics such as the Bergman metric and the extrinsic geometry associated to $g$ in case $M$ is biholomorphic to a bounded domain in $\mathbb C^n$ (e.g. the unit ball), which is guaranteed by the Bishop-Narasimhan-Remmert theorem.  We shall prove the following

\begin{theorem}\label{th:B-Metric}
Suppose that $M$ is biholomorphic to a bounded pseudoconvex domain with $C^2$ boundary  in $\mathbb C^n$ with $n>1$.  Then
$$
\liminf_{\rho\rightarrow \infty} \rho^2\left[ (1-\varepsilon){\rm Tr}_g(\omega_B) + (n+1)\, {\rm Scal}_g \right]  \le -(n-1)n(n+1)
$$
for all $\varepsilon>0$, where $\omega_B$ stands for the K\"ahler form of the Bergman metric on $M$. 
\end{theorem}

Let $\mathcal{H}^{p,q}_{L^2}(M,L)$ be the reduced $L^2$ cohomology group associated to $M$ and $L$.  The Bochner-Kodaira-Nakano inequality combined with a trick going back to Donnelly-Fefferman \cite{DonnellyFefferman} gives

\begin{theorem}\label{cor:finite_cohomology}
Suppose either of the following condition holds
\begin{enumerate}
\item $p+q<n-2$ and there exists $0<\alpha<\frac{2(n-p-q-2)}{n-p-q}$ such that
$
\Theta_h\leq_{\mathrm{as}}\alpha\cdot\omega/\rho^2,
$
\item $p+q>n+2$ and there exists $0<\alpha<\frac{2(p+q-n-2)}{p+q-n}$ such that
$
\Theta_h\geq_{\mathrm{as}}-\alpha\cdot\omega/\rho^2.
$
\end{enumerate}
Then  $\dim_\C\mathcal{H}^{p,q}_{L^2}(M,L)<\infty$.
\end{theorem}

\begin{remark}
{\rm In particular, $\dim_\C {H}^{0}_{L^2}(M,L)<\infty$ if $n>2$ and $
\Theta_h\leq_{\mathrm{as}}\alpha\cdot\omega/\rho^2
$ with $\alpha<\frac{2(n-2)}n$. It is remarkable that the better bound  $\alpha<n-1$ is already sufficient if one applies  Theorem \ref{th:finite_0}. }
\end{remark}

Without curvature assumption we have the following finiteness result on $L^p$ plurigenera.

 \begin{theorem}\label{th:affine}
 If $M$ is a smooth affine-algebraic variety, then $P_{m,L^p}(M)<\infty$ for all $p$ and $m$.
  \end{theorem}

Finiteness of $P_{m,L^{2/m}}(M)$  was first proved by  Sakai \cite{Sakai}  by using a smooth compactification of $M$ (which exists, by Hironaka's resolution of singularities). It is not clear whether his method  works for  general $L^p$ plurigenera, since $P_{m,L^{2/m}}(M)$ depends only on the complex structure. Here we shall present a self-contained analytic approach to  Theorem \ref{th:affine} by using Demailly's calculation of $dV_g$ for affine-algebraic varieties (cf. \cite{DemaillyAffine}).

 \begin{corollary}\label{cor:affine}
If\/ $M$ is a smooth affine-algebraic variety, then the Ricci curvature of $g$ satisfies
$$
\limsup_{\rho\rightarrow \infty} \rho^2\cdot {\rm Ric}_g = 0.
$$
\end{corollary}
\begin{remark}
{\rm We refer to \cite{SasakiShiga},  \cite{Vitter} and \cite{Yang} for more results on growth of curvatures for complex hypersurfaces in $\mathbb C^N$.}
\end{remark}

Some natural questions arise.

 \begin{problem}\label{prob:Affine_1}
{\rm Let $M\hookrightarrow {\mathbb C}^N$ be a closed $1-$dimensional complex submanifold. Is it possible to conclude that $M$ is affine-algebraic if and only if $P_{m,L^2}(M)<\infty$ for all $m\ge 1$?}
  \end{problem}

  \begin{problem}\label{prob:algebraic}
 {\rm Let $M\hookrightarrow {\mathbb C}^N$ be a non-algebraic complex submanifold of dimension $n>1$. Is it possible to conclude that either $P_{m,L^2}(M)=0$ for all $m$ or $P_{m,L^2}(M)=\infty$ for some $m$?}
  \end{problem}

A partial answer to Problem \ref{prob:Affine_1} is given as follows.

 \begin{theorem}\label{th:plurigenera_curve}
  Let $f$ be an entire function in $\mathbb C$ and  $\Gamma(f)$ the graph of $f$ in $\mathbb C^2$. Then $
 \Gamma(f)
 $
 is affine-algebraic if and only if $P_{m,L^2}(\Gamma(f))<\infty$ for all $m\ge 1.
  $
  Moreover, if $f$ is transcendental then $P_{3,L^2}(\Gamma(f))=\infty$.
 \end{theorem}

 \begin{remark}
  {\rm (1) In particular, we have plenty of Riemann surfaces embedded in $\mathbb C^2$ which are\/ {\it biholomorphic}\/ to $\mathbb C$ but not all $L^2$ plurigenera are finite.}

{\rm (2) It may happen that $P_{m,L^2}(M)=0$ for all $m>0$ when $M$ is non-algebraic and $n>1$. For instance, we may consider the cylinder $M=\mathbb{C}\times \Gamma(f)$ in $\mathbb C^3$, where $f$ is a transcendental entire function in $\mathbb C$, in view of  Theorem \ref{th:plurigenera_curve}.}

{ \rm (3) Mok \cite{MokAffine} and Demailly \cite{DemaillyAffine} gave geometric and complex-analytic conditions for a complete K\"ahler manifold to be\/ {\it biholomorphic}\/ to an affine-algebraic variety.}

 \end{remark}

The methods of this paper also work for investigating finiteness theorems on harmonic functions of minimal submanifolds in $\mathbb R^N$. We shall return to this subject in a future paper.

\section{Preliminaries}
\subsection{Some differential inequalities} Let $X$ be an $n-$dimensional Riemannian manifold with a metric given locally by $g=\sum_{j,k} g_{jk}dx_j dx_k$. Let $(g^{jk})=(g_{jk})^{-1}$.  The Laplace operator is given by
 $$
 \Delta =\frac1{\sqrt{{\rm det}(g_{jk})}} \sum_{j,k} \frac{\partial}{\partial x_j} \left[\sqrt{{\rm det}(g_{jk})} g^{jk}\frac{\partial}{\partial x_k}\right].
  $$
    Let $d V=\sqrt{{\rm det}(g_{jk})} dx_1\cdots dx_n$ be the Riemannian volume element. For real-valued functions $u\in C^1_0(X)$ and $v\in C^2(X)$, integral by parts gives
  \begin{equation}\label{eq:GreenFormula}
  \int_X u \Delta v dV =-\int_X \nabla u \nabla v dV
  \end{equation}
  where
  $$
  \nabla u=\left(\sum_j g^{1j}\frac{\partial u}{\partial x_j}, \cdots, \sum_j g^{nj}\frac{\partial u}{\partial x_j}\right),
  $$
  $$
  \nabla u \nabla v=\sum_{j,k} g^{jk} \frac{\partial u}{\partial x_j} \frac{\partial v}{\partial x_j}.
    $$
    Set $|\nabla u|^2=\nabla u \nabla u$. Then we have

   \begin{proposition}\label{prop:Hardy-Sobolev}
    Let $\eta:{\mathbb R}\rightarrow [0,\infty)$ be a $C^1$ function with $\eta'>0$. Let $0<\gamma< 1$. For  $\phi\in C^1_0(X)$ and $\psi\in C^2(X)$ we have
   \begin{equation}\label{eq:Laplace+}
   \int_X \phi^2 \left[ \eta (\psi)\Delta\psi + (1-\gamma) \eta'(\psi){|\nabla\psi|^2}\right] dV\le \frac1{\gamma} \int_X \frac{\eta^2(\psi)}{\eta'(\psi)} |\nabla\phi|^2 dV
  \end{equation}
  and
  \begin{equation}\label{eq:Laplace-}
   \int_X {\phi^2}\left[\frac{\Delta\psi}{\eta (-\psi)} +(1-\gamma) \frac{\eta'(-\psi)}{\eta^2(-\psi)}{|\nabla\psi|^2}\right] dV \le \frac1{\gamma} \int_X \frac{|\nabla\phi|^2}{\eta'(-\psi)} dV.
  \end{equation} \end{proposition}

 \begin{proof}
 By $(\ref{eq:GreenFormula})$ we have
 \begin{eqnarray*}
  \int_X \phi^2 \eta (\psi) \Delta\psi dV & = & -\int_X \nabla \psi  \nabla \left(\eta(\psi)\phi^2\right)dV\\
&  = & -2 \int_X \phi \eta(\psi)\nabla \psi \nabla \phi dV - \int_X \phi^2 \eta'(\psi)|\nabla \psi|^2 dV,
\end{eqnarray*}
so that
\begin{eqnarray*}
 && \int_X \phi^2\eta (\psi) \Delta\psi dV+\int_X \phi^2 \eta'(\psi) |\nabla \psi|^2 dV\\
 & = &  -2 \int_X \phi \eta(\psi)\nabla \psi \nabla \phi dV\\
 & \le & \gamma  \int_X  \phi^2 \eta'(\psi)|\nabla \psi|^2 dV+ \frac1{\gamma} \int_X \frac{\eta^2(\psi)}{\eta'(\psi)}|\nabla\phi|^2 dV
\end{eqnarray*}
in view of the Schwarz inequality, i.e., (\ref{eq:Laplace+}) holds.

Applying $(\ref{eq:Laplace+})$ with $\eta(t)$ replaced by $1/\eta(-t)$, we immediately obtain $(\ref{eq:Laplace-})$.
\end{proof}

\begin{remark}
{\rm By taking $\eta(t)=t$ in $(\ref{eq:Laplace+})$, we immediately obtain the following Caccioppoli-type inequality}
 \begin{equation}\label{eq:CacciopoliIneq}
\int_{X} \phi^2  {|\nabla\psi|^2} dV \le \frac1{\gamma(1-\gamma)} \int_{X} \psi^2 |\nabla\phi|^2 dV - \frac1{1-\gamma}\int_{X}\phi^2 \psi \Delta\psi dV.
 \end{equation}

  \end{remark}

\subsection{A finiteness theorem} In order to deal with the $L^p$ Bergman space for $0<p<1$, we need to extend a classical finiteness theorem of F. Riesz for normed vector spaces to more general cases.

 \begin{definition}
{\rm Let $E$ be a complex vector space. A\/ {\it quasi-norm}\/ ${\rm q}$ for $E$ is a mapping from $E$ into $[0,\infty)$ such that}
\begin{enumerate}
{\rm \item ${\rm q}(x)=0$ if and only if $x=0$,
\item ${\rm q}(cx)=|c|^\alpha {\rm q}(x)$ for  $c\in {\mathbb C}$,  $x\in E$ and some constant $\alpha>0$,
\item ${\rm q}(x+y)\le {\rm q}(x)+{\rm q}(y)$ for  $x,y\in E$.}
\end{enumerate}
\end{definition}

Notice that $d(x,y):={\rm q}(x-y)$ defines a metric for $E$. Equipped with the metric topology, $E$ is called a quasi-normed space, which is naturally  a\/ {\it topological}\/ vector space. It follows from Theorem 1.21 and Theorem 1.22 in \cite{RudinBook} that

\begin{proposition}\label{lm:homeom}
Every $n-$dimensional quasi-normed space is homeomorphic to ${\mathbb C}^n$.
\end{proposition}

\begin{proposition}\label{prop:finitedimen}
Let $(E,{\rm q})$ be a quasi-normed space. Then the following properties are equivalent:
\begin{enumerate}
\item $E$ is finite-dimensional;
\item  Every sequence in the closed ball $B:=\{x\in E: {\rm q}(x)\leq1\}$ has a convergent subsequence;
\item  Every sequence in the unit sphere $S:=\{x\in E: {\rm q}(x)=1\}$ has a convergent subsequence.
\end{enumerate}
\end{proposition}

\section{Proofs of Theorem \ref{th:finite_0} and Theorem \ref{th:B-Metric}} 

 We begin with two elementary lemmas.

\begin{lemma}\label{lm:ddbar}
For every $C^2$ function $\varphi$ on ${\mathbb C}^N$ we have
$$
i\partial\bar{\partial} (\varphi|_M) = (i\partial\bar{\partial} \varphi)|_M.
$$
\end{lemma}

\begin{proof}
On every sufficiently small domain $U$ in ${\mathbb C}^N$ with $U\cap M\neq \emptyset$ we choose local coordinates $(\zeta_1,\cdots,\zeta_n,\cdots,\zeta_N)$
such that
$$
M\cap U=\{(\zeta_1,\cdots,\zeta_n,0,\cdots,0)\}.
$$
Then we have $\varphi|_M=\varphi(\zeta_1,\cdots,\zeta_n,0,\cdots,0)$ so that
$$
i\partial\bar{\partial} (\varphi|_M) = i \sum_{j,k=1}^n \frac{\partial^2 \varphi}{\partial \zeta_j\partial\bar{\zeta}_k}(\zeta_1,\cdots,\zeta_n,0,\cdots,0)d\zeta_j\wedge d\bar{\zeta}_k.
$$
On the other hand, since
$$
i\partial\bar{\partial} \varphi= i \sum_{j,k=1}^N \frac{\partial^2 \varphi}{\partial \zeta_j\partial\bar{\zeta}_k}d\zeta_j\wedge d\bar{\zeta}_k,
$$
we have
\[(i\partial\bar{\partial} \varphi)|_M = i \sum_{j,k=1}^n \frac{\partial^2 \varphi}{\partial \zeta_j\partial\bar{\zeta}_k}(\zeta_1,\cdots,\zeta_n,0,\cdots,0)d\zeta_j\wedge d\bar{\zeta}_k.\qedhere\]
\end{proof}

\begin{lemma}[Hardy inequality]\label{lm:Hardy}
Suppose $n>1$. For every  real-valued $C^1$ function $\phi$ with compact support in $M$, we have
\begin{equation}\label{eq:Hardy}
(n-1)^2 \int_M {\phi^2}/{\rho^2} dV \le \int_M |\nabla \phi|^2 dV.
\end{equation}
\end{lemma}

\begin{proof}
 We write
 $
 g= \sum_{j,k=1}^n g_{j\bar{k}}d\zeta_j d\bar{\zeta}_k
 $
 in local coordinates of $M$. Let $(g^{\bar{j}{k}})$ be the inverse matrix of $(g_{j\bar{k}})$, that is
 $$
 \sum_{k=1}^n g_{j\bar{k}}g^{\bar{k}l}=\delta_{jl}=\sum_{k=1}^n g^{\bar{j}k} g_{k\bar{l}}
 $$
 where $\delta_{jl}$ is the Kronecker delta.
  The complex Laplacian  is given by
   $$
   \Box=\sum_{j,k=1}^n g^{\bar{k}{j}} \partial^2/\partial \zeta_j\partial\bar{\zeta}_k.
      $$
     Clearly we have $\Delta=4\Box$. By Lemma \ref{lm:ddbar} we have
      $$
     g_{j\bar{k}} = \partial^2 \rho^2/{\partial \zeta_j\partial \bar{\zeta_k}},
      $$
      so that
 \begin{equation}\label{eq:Basic}
\Box \rho^2= \sum_{j,k=1}^n g_{j\bar{k}}g^{\bar{k}j} = \sum_{j=1}^n 1  = n.
\end{equation}
Set $\psi=-\rho^{-2n+2}$. A direct calculation shows
$$
\nabla \psi = 2(n-1) \rho^{-2n+1}\nabla \rho
$$
and
\begin{eqnarray*}
\Delta \psi  & = & (n-1)\left[\rho^{-2n}\Delta \rho^2-n \rho^{-2(n+1)} |\nabla \rho^2|^2\right]\\
& = & 4n(n-1)\rho^{-2n}(1-|\nabla \rho|^2).
\end{eqnarray*}
Applying $(\ref{eq:Laplace-})$ with $\eta(t)=t$ and $\gamma=1/2$, we have
\begin{eqnarray*}
\int_M |\nabla \phi|^2 dV & \ge & \frac14 \int_M \phi^2 \left[\frac{2\Delta \psi}{-\psi}+\frac{|\nabla \psi|^2}{\psi^2}\right] dV\\
& = & \frac14
                    \int_M \frac{\phi^2}{\rho^2} \left[8n(n-1)(1-|\nabla \rho|^2)+4(n-1)^2|\nabla \rho|^2\right] dV\\
  & = & (n-1) \int_M \frac{\phi^2}{\rho^2} \left[2n-(n+1)|\nabla \rho|^2\right] dV\\
  &  \ge & (n-1)^2 \int_M \frac{\phi^2}{\rho^2} dV
\end{eqnarray*}
 for $|\nabla \rho|\le 1$.
\end{proof}

\begin{proof}[Proof of Theorem \ref{th:finite_0}]
By $(\ref{eq:Basic})$ we have
 $$
 \Delta \log \rho^2=\frac{\Delta \rho^2}{\rho^2}-\frac{|\nabla \rho^2|^2}{\rho^4}=\frac{4\Box \rho^2}{\rho^2}-\frac{4|\nabla \rho|^2}{\rho^2}\ge \frac{4(n-1)}{\rho^2}.
 $$
 Let $s$ be a  holomorphic section of $L$ over $M$. Set $\varphi= \log \rho^2$. For every $p>0$ we have
      \begin{eqnarray*}
      \Delta\left[\frac{p}4\log |s|^2_{h}+\frac{\varphi}2\right] & = & p\,\Box \log |s|^2_{h}+\frac{\Delta \varphi}2\\
      & \ge & -p {\rm Tr}_g(\Theta_h)  + \frac{2(n-1)}{\rho^2}\\
      & \ge & \frac{2 (n-1)-\alpha p}{\rho^2}
      \end{eqnarray*}
whenever  $\rho\ge b\gg 1$, and
            $$
            \Delta\left[\frac{p}4\log |s|^2_{h}+\frac{\varphi}2\right]\ge -{\rm const}/\rho^2
                        $$
                        on whole $M$ (where the constant might depend on $\alpha,p$ but is independent of $s$).
             Set
             $$
             \psi:=\exp\left[\frac{p}4\log |s|^2_{h}+\frac{\varphi}2\right].
             $$
              Then we have
            \begin{equation}\label{eq:Delta_1}
            \Delta \psi \ge  [2 (n-1)-\alpha p]\cdot \psi/\rho^2 + |\nabla \psi|^2/\psi
            \end{equation}
              on $\{\rho>b\}\cap \{\psi>0\}$, and
              \begin{equation}\label{eq:whole}
              \Delta \psi \ge -{\rm const}\cdot \psi/\rho^2
              \end{equation}
              on whole $M$. Notice also that the set $\{\psi=0\}$ is of  measure $0$ and $\psi$ is smooth outside this set. 

Let $\sigma$ be a smooth convex nondecreasing function on ${\mathbb R}$ such that $\sigma|_{(-\infty,1/4]}=1/2$, $\sigma (t)=t$  for $t\ge 1$ and $\sigma'\leq1$. For $\delta>0$ we set $\psi_\delta:=\delta\sigma(\psi/\delta)$.
Then $\psi_\delta$ is a nonnegative smooth psh function on $M$. Moreover, we have
  $$
  \nabla \psi_\delta= \sigma'(\psi/\delta) \nabla\psi,
  $$
and  $\nabla\psi_\delta=\nabla\psi$ when $\psi>\delta$.
Since $\sigma'\leq1$, it follows from \eqref{eq:whole} that
\begin{equation}\label{eq:DeltaPsidelta}
\Delta\psi_\delta=\sigma'(\psi/\delta)\Delta\psi+\sigma''(\psi/\delta)|\nabla\psi|^2/\delta\geq-\mathrm{const}\cdot\psi/\rho^2.
\end{equation}
By \eqref{eq:CacciopoliIneq} and \eqref{eq:DeltaPsidelta}, we have
\begin{eqnarray*}
&& \int_{\{\psi>\delta\}}\phi^2|\nabla\psi|^2dV \leq \int_M\phi^2|\nabla\psi_\delta|^2dV\\
& \leq & \frac{1}{\gamma(1-\gamma)}\int_M\psi^2_\delta|\nabla\phi|^2dV-\frac{1}{1-\gamma}\int_M\phi^2\psi_\delta\Delta\psi_\delta{dV}\\
& \leq & \frac{1}{\gamma(1-\gamma)}\int_M\psi^2_\delta|\nabla\phi|^2dV+\frac{\mathrm{const}}{1-\gamma}\int_M\phi^2\psi_\delta\psi/\rho^2{dV}
\end{eqnarray*}
for all  $\phi\in C^1_0(M)$.  Letting $\delta\rightarrow0$, we conclude that $\nabla\psi\in{L}^2_{\mathrm{loc}}$.

   Let $\chi:{\mathbb R}\rightarrow [0,1]$ be a cut-off function such that $\chi|_{(-\infty,1]}=1$,  $\chi|_{[2,\infty)}=0$ and $|\chi'|\le 2$. We set
  $$
  \phi=[1-\chi(\rho/b)]\chi(\rho/R),\ \ \ R\gg b.
  $$
 and we have
\begin{eqnarray}\label{eq:KeyEstimate}
(n-1)^2\int_{2b<\rho<R} {\psi_\delta^2}/{\rho^2} dV
& \le &  (n-1)^2\int_{M} {(\phi\psi_\delta)^2}/{\rho^2} dV\nonumber\\
 & \le & \int_M |\nabla (\phi \psi_\delta)|^2 dV\nonumber\ \ \ \ \ \ \ \ ({\rm by\ }(\ref{eq:Hardy}))\\
 & \le & \int_{2b<\rho < R} |\nabla \psi_\delta|^2dV
            +2\int_{M\backslash\{2b<\rho < R\}} \phi^2|\nabla \psi_\delta|^2dV\nonumber\\
 &&      +2\int_{M\backslash\{2b<\rho < R\}} \psi_\delta^2|\nabla \phi |^2dV\nonumber\\
 & \le &  \int_{2b<\rho < R} |\nabla \psi_\delta|^2dV +2\int_{b<\rho<2b} |\nabla\psi_\delta|^2   dV\\
  & &  + 2 \int_{R<\rho<2R} |\nabla\psi_\delta|^2   dV +2\int_{M} \psi_\delta^2|\nabla \phi |^2dV\nonumber.
\end{eqnarray}
Fix $0<\gamma< 1$ for a moment. By $(\ref{eq:CacciopoliIneq})$ we have
  \begin{eqnarray}\label{eq:GradientEstimate}
  && \int_{2b<\rho<R} |\nabla\psi_\delta|^2   dV \le  \int_M \phi^2|\nabla\psi_\delta|^2   dV\nonumber\\
 &  \le & \frac1{\gamma(1-\gamma)} \int_M \psi_\delta^2 |\nabla \phi|^2 dV - \frac1{1-\gamma} \int_M \phi^2 \psi_\delta \Delta \psi_\delta dV\nonumber\\
  &  \le & \frac1{\gamma(1-\gamma)} \cdot \frac{{2}}{b^2} \int_{b<\rho<2b} \psi_\delta^2  dV+\frac1{\gamma(1-\gamma)} \cdot\frac{{2}}{R^2} \int_{R<\rho<2R} \psi_\delta^2  dV\\
  &  & - \frac1{1-\gamma} \int_M \phi^2 \psi_\delta \Delta \psi_\delta dV\nonumber.
  \end{eqnarray}
Let $\kappa:{\mathbb R}\rightarrow [0,1]$ be a cut-off function such that $\kappa=1$ on $[1,2]$ and $\kappa=0$ on $(-\infty,1/2]\cup [3,\infty)$. Analogously we have
 \begin{eqnarray}\label{eq:GradientEstimate_2}
 \int_{b<\rho<2b} |\nabla\psi_\delta|^2   dV
 &  \le & \frac1{\gamma(1-\gamma)} \int_M \psi_\delta^2 |\nabla \kappa(\rho/b)|^2 dV\nonumber  \\
 & &- \frac1{1-\gamma} \int_M \kappa(\rho/b)^2 \psi_\delta \Delta \psi_\delta dV\nonumber\\
 & \le & \frac{C_\gamma}{b^2} \int_{b/2<\rho<3b}   \psi_\delta^2 dV \\
 & &- \frac1{1-\gamma} \int_M \kappa(\rho/b)^2 \psi_\delta \Delta \psi_\delta dV\nonumber,
  \end{eqnarray}
  and
   \begin{eqnarray}\label{eq:eq:GradientEstimate_3}
 \int_{R<\rho<2R} |\nabla\psi_\delta|^2   dV & \le & \frac{C_\gamma}{R^2} \int_{R/2<\rho<3R}   \psi_\delta^2 dV\\
 &&  - \frac1{1-\gamma} \int_M \kappa(\rho/R)^2 \psi_\delta \Delta \psi_\delta dV\nonumber
  \end{eqnarray}
  where $C_\gamma$ denotes a generic constant depending only on $\gamma$.

 Notice that if $\delta/4\le \psi\le \delta$ then $\delta/2\le \psi_\delta\le \delta$ (e.g., $\psi_\delta\asymp \psi$). It follows from \eqref{eq:DeltaPsidelta} that

  \begin{equation}\label{eq:Delta_3}
  \int_M \kappa(\rho/b)^2 \psi_\delta \Delta \psi_\delta dV \ge -{\rm const}\cdot \int_{b/2<\rho<3b}\psi^2/\rho^2 dV
  \end{equation}
        and
  \begin{equation}\label{eq:Delta_4}
  \int_M \kappa(\rho/R)^2 \psi_\delta \Delta \psi_\delta dV \ge -{\rm const}\cdot \int_{R/2<\rho<3R} \psi^2/\rho^2 dV.
        \end{equation}
 We also have
  \begin{align*}
  \int_M \phi^2 \psi_\delta \Delta \psi_\delta dV  = & \int_{\psi>\delta}  \phi^2 \psi_\delta \Delta \psi_\delta dV + \int_{\delta/4\le \psi\le \delta}  \phi^2 \psi_\delta \Delta \psi_\delta dV\\
    \ge & -|2 (n-1)-\alpha p| \int_{\{b<\rho<2R\}\cap \{\psi>\delta\}} \psi^2/\rho^2 dV\\
    & + \int_{\{2b<\rho<R\}\cap \{\psi>\delta\}} |\nabla \psi|^2 dV\ \ \ \ \ \ \ \  ({\rm by\ \ }(\ref{eq:Delta_1}))\\
    & -{\rm const} \cdot \int_{\{\delta/4\le \psi\le \delta\}\cap\, {\rm supp\,}\phi}\psi^2/\rho^2 dV\\
    \ge & -|2 (n-1)-\alpha p| \int_{\{b<\rho<2R\}\cap \{\psi>\delta\}} \psi^2/\rho^2 dV\\
    & + \int_{2b<\rho<R} |\nabla \psi_\delta|^2 dV-\int_{\{2b<\rho<R\}\cap \{\psi\le \delta\}}|\nabla\psi_\delta|^2 dV\\
    & -{\rm const} \cdot \int_{\{\delta/4\le \psi\le \delta\}\cap\, {\rm supp\,}\phi}\psi^2/\rho^2 dV.
  \end{align*}
           This combined with $(\ref{eq:GradientEstimate})$ yields
   \begin{eqnarray}\label{eq:GradientEstimate_4}
 \int_{2b<\rho<R} |\nabla\psi_\delta|^2   dV  & \le &  \frac{C_\gamma}{b^2} \int_{b<\rho<2b} \psi_\delta^2  dV +\frac{C_\gamma}{R^2} \int_{R<\rho<2R} \psi_\delta^2  dV \\
   & & + \frac{|2 (n-1)-\alpha p|}{2-2\gamma}\int_{\{b<\rho<2R\}\cap \{\psi>\delta\}} \psi^2/\rho^2 dV\nonumber\\
    & & + C_\gamma \int_{\{2b<\rho<R\}\cap \{\psi\le \delta\}}|\nabla\psi|^2 dV\nonumber\\
    & & + C_\gamma \int_{\{\delta/4\le \psi\le \delta\}\cap\, {\rm supp\,}\phi}\psi^2/\rho^2 dV\nonumber.
    \end{eqnarray}
    Applying $(\ref{eq:KeyEstimate})$, $(\ref{eq:GradientEstimate_2})\sim (\ref{eq:GradientEstimate_4})$ by first letting $\delta\rightarrow 0$ then letting $R\rightarrow \infty$, we obtain
 \begin{eqnarray*}
&& C_\gamma  \int_{b/2<\rho<3b} \psi^2/\rho^2  dV\\
& \ge & (n-1)^2\int_{\rho>2b} {\psi^2}/{\rho^2} dV
    -\frac{|2 (n-1)-\alpha p|}{2-2\gamma} \int_{\rho>b} \psi^2/\rho^2 dV
 \end{eqnarray*}
whenever
     $$
      \int_M {\psi^2}/{\rho^2} dV= \int_M |s|^p_{h}  dV<\infty
                $$
                (recall that $|\nabla \psi|\in L^2_{\rm loc}$).

                If $p<\frac{2n(n-1)}{\alpha}$ and $\gamma\ll1 $ then we have
                $$
                \varepsilon:= (n-1)^2-\frac{|2 (n-1)-\alpha p|}{2-2\gamma}>0.
                                $$
                                      Thus
  $$
 \varepsilon\cdot \int_M {\psi^2}/{\rho^2} dV \le {\rm const}\cdot\int_{{\rho<3 b}} {\psi^2}/{\rho^2} dV,
    $$
    i.e.,
    \begin{equation}\label{eq:KeyIneq_0}
    \varepsilon\cdot\int_M |s|^p_{h}   dV \le{\rm const}\cdot \int_{\rho<3b} |s|^p_{h} dV.
        \end{equation}
We define a quasi-norm ${\rm q}$ on $H^0_{L^p}(M,L)$  by
        $$
        {\rm q}(s)=\left\{
        \begin{array}{ll}
        \left[\int_M |s|^p_{h}  dV\right]^{1/p}  & p\ge 1\\
        \int_M |s|^p_{h}  dV & 0<p<1
        \end{array}
                      \right.
        $$
        (notice that ${\rm q}$ is even a norm for $p\ge 1$). By Proposition \ref{prop:finitedimen}, it suffices to verify that every sequence  $\{s_j\}\subset H^0_{L^p}(M,L)$  with ${\rm q}(s_j)\le1$ contains a convergent subsequence.
          Let $\{U_m\}^\infty_{m=1}$ be a locally finite covering of $M$ such that $L$ is  trivial over $U_m$   for all $m$. Let  $\xi_m$ be a frame of $L$ on $U_m$. We may write $s_j=s^*_{m,j}\xi_m$ for some holomorphic function $s^*_{m,j}$ on $U_m$, and $|s_j|_h=|s^*_{m,j}||\xi_m|_h$. Since $\mathrm{q}(s_j)\leq1$ and $|\xi_m|_h$ is a nonzero smooth function, it follows from the mean value inequality that $\{s^*_{m,j}\}^\infty_{j=1}$ is a sequence of locally uniformly bounded holomorphic functions on $U_m$. Montel's theorem combined with a diagonal process yields that  there exists a subsequence $\{j_k\}$ which is independent of $m$, such that $\{s^*_{m,j_k}\}^\infty_{k=1}$ converging locally uniformly to a holomorphic function $s^*_m$ on $U_m$. Therefore, the subsequence $\{s_{j_k}\}^\infty_{k=1}$ converges locally uniformly (with respect to the metric $h$) to some holomorphic section $s$ of $L$ with $s|_{U_m}=s^*_m\xi_m$. By the Fatou lemma we have
        $$
        \int_{M} |s|^p_{h}  dV
        \le \liminf_{k\rightarrow \infty}\int_{M} |s_{j_k}|^p_{h}  dV \le1,
                       $$
                       i.e., ${\rm q}(s)\le 1$. On the other hand,  $(\ref{eq:KeyIneq_0})$ implies that
                       $$
                         {\rm q}(s_{j_k}-s) \rightarrow 0
                                             $$
                                             as $k\rightarrow \infty$. Thus we have verified that $H^0_{L^p}(M,L)$ is finite-dimensional.

Finally, if\/ $ {\rm Tr}_g(\Theta_h) \le  \alpha/\rho^2$ holds on whole $M$, then we may take $b=0$ in (\ref{eq:KeyIneq_0}) so that $s=0$ for any $s\in H^0_{L^p}(M,L)$.
  \end{proof}

\begin{proof}[Proof of Theorem \ref{th:B-Metric}]
Let $F:M\rightarrow \Omega$ be a  biholomorphic mapping where $\Omega$ is a bounded pseudoconvex domain with $C^2$ boundary in $\mathbb C^n$.  Let $\mathcal K_\Omega(z)$ be the standard Bergman kernel on $\Omega$. We consider the following weighted Bergman space
$$
A^2_\alpha(\Omega)=\left\{f\in \mathcal O(\Omega): \int_\Omega |f|^2 \mathcal K_\Omega^\alpha<\infty \right\},\ \ \ \alpha > 0.
$$
It is easy to verify that
$$
\mathcal K_\Omega(z) \lesssim \delta_\Omega(z)^{-n-1}
$$
where $\delta_\Omega$ is the Euclidean boundary distance of $\Omega$. Since $\delta_\Omega^{-r}$ is integrable on $\Omega$ for each $r<1$, it follows from the previous inequality that  $A^2_\alpha(\Omega)$ contains all complex polynomials for each $\alpha<\frac1{n+1}$, which is of infinite-dimensional.
Suppose on the contrary that
$$
\liminf_{\rho\rightarrow \infty} \rho^2\left[ (1-\varepsilon){\rm Tr}_g(\omega_B) + (n+1) {\rm Scal}_g \right]  > -(n-1)n(n+1).
$$
for some $\varepsilon>0$. Then there exist $\alpha<\frac1{n+1}$ and $\beta<1$ such that 
$$
-\alpha {\rm Tr}_g(\omega_B)- {\rm Scal}_g \le_{\rm as} \frac{\beta n(n-1)}{\rho^2}.
$$
 Let $K_M$ be the canonical bundle of $M$ and let $h$ be the Hermitian metric on $K_M$ which is given by $(dV_g)^{-1} \cdot F^\ast (\mathcal K_\Omega^\alpha)$. Then we have $H^0_{L^2}(M,K_M)\simeq A^2_\alpha(\Omega)$. Note that 
$$
\Theta_h = -{\rm Ric}_g -\alpha i\partial \bar{\partial} \log F^\ast (\mathcal K_\Omega)=-{\rm Ric}_g -\alpha \omega_B
$$
since the Bergman metric is invariant under biholomorphic mappings (we refer to \cite{Kobayashi} for the definition of the Bergman metric on complex manifolds). Thus
$$
{\rm Tr}_g(\Theta_h) =- {\rm Scal}_g -\alpha {\rm Tr}_g(\omega_B) \le_{\rm as} \frac{\beta n(n-1)}{\rho^2}.
$$
In view of Theorem \ref{th:finite_0}, we conclude that $H^0_{L^2}(M,K_M)$
 is of finite-dimensional, so is $A^2_\alpha(\Omega)$, a contradiction.
\end{proof}
\section{Proof of Theorem \ref{th:finite}}

Let $\lambda:(0,\infty)\rightarrow (0,\infty)$ be a continuous increasing function such that
$
\int_1^\infty dt/\lambda(t)<\infty.
$
Fix $\delta>0$.
Let $a=a(\delta)>1$ satisfy
\begin{equation}\label{eq:Integ_1}
\int_{a^2}^\infty \frac{dt}{\lambda(t)}\le \delta.
\end{equation}
Let
$b=b(\delta)$ satisfy
\begin{equation}\label{eq:Integ_2'}
b\ge 2(a+1),
\end{equation}
\begin{equation}\label{eq:Integ_2}
2 \int_{b^2/2}^\infty \frac{dt}{\lambda(t)}\le \int_{(a+1)^2}^{2(a+1)^2} \frac{dt}{\lambda(t)}.
\end{equation}
We first show the following

\begin{lemma}\label{lm:Construction}
There exists a nonnegative $C^2$ psh function $\varphi$ on ${\mathbb C}^N$ such that for $|z|\ge b$:
\begin{enumerate}
\item $\varphi(z)\le \delta \log |z|^2$;
\item $i\partial\bar{\partial}\varphi\ge [1/\lambda(|z|^2)] {i\partial\bar{\partial}|z|^2}$.
\end{enumerate}
\end{lemma}

  \begin{proof}
   Let
$\vartheta$ be a nonnegative continuous function on ${\mathbb R}$ satisfying  $\vartheta|_{(-\infty,a^2]}=0$,  $\vartheta=1/\lambda$ on $[(a+1)^2,\infty)$, and $0< \vartheta\le 1/\lambda$ on $(a^2,(a+1)^2)$. Following \cite{GreeneWuBook}, p.\,184--185, we define
  $$
  \kappa(t)=\int_{0}^t \frac{dr_1}{r_1} \int_{0}^{r_1}\vartheta(r_2) {dr_2},\ \ \ t\ge 0.
  $$
  Clearly, $\kappa$ is $C^2$ and satisfies
  $$
  \kappa'(t)+t\kappa''(t)=1/\lambda(t),\ \ \ t\ge (a+1)^2.
  $$
  Since  $\lambda$ is increasing, it follows that for $t\ge b^2$,
 $$
 \frac{t}{\lambda(t)}\le 2 \int_{t/2}^{t} \frac{dr}{\lambda(r)}\le 2 \int_{b^2/2}^{\infty} \frac{dr}{\lambda(r)}\le \int_0^t \vartheta(r)dr
  $$
  in view of $(\ref{eq:Integ_2})$.
  Thus for  $t\ge b^2$,
  $$
  \kappa''(t)=\frac1{t^2}\left[\frac{t}{\lambda(t)}-\int_0^t \vartheta(r)dr \right]\le 0.
  $$
   Set $\varphi(z)=\kappa(|z|^2)$, $z\in {\mathbb C}^N$. If $|z|\ge b$, then
  \begin{eqnarray}\label{eq:Hessian}
  i\partial \bar{\partial} \varphi & = & \kappa'(|z|^2)i\partial\bar{\partial}|z|^2+\kappa''(|z|^2)i\partial |z|^2\wedge \bar{\partial} |z|^2\nonumber\\
  & \ge & [\kappa'(|z|^2)+|z|^2\kappa''(|z|^2)]i\partial\bar{\partial}|z|^2\nonumber\\
  & = & \frac{i\partial\bar{\partial}|z|^2}{\lambda(|z|^2)}.
  \end{eqnarray}
  Since
 $
 \kappa(t)\le  \delta \log t
 $
 for $t\ge a^2$ in view of $(\ref{eq:Integ_1})$,
  it follows that
  $
  \varphi(z)\le \delta \log |z|^2
  $
  whenever $|z|\ge b> a$.
  \end{proof}

        \begin{proof}[Proof of Theorem \ref{th:finite}]
          We define
$$
{\lambda}(t):= \min\{ \mu(t),(1+t)\left[\log (3+t)\right]^2\},\ \ \ t\ge 0.
$$
Clearly, $\lambda$ is a continuous increasing function with $\int_1^\infty dt/\lambda(t)<\infty$.
      Applying Lemma \ref{lm:Construction} with $\lambda$ replaced by $\lambda/(p+1)$, we obtain a nonnegative $C^2$ psh function $\varphi$ on ${\mathbb C}^N$ such that
      $\varphi(z)\le (2-\alpha) \log |z|$ and
 $$
 i\partial\bar{\partial}\varphi\ge \frac{p+1}{\lambda(|z|^2)} {i\partial\bar{\partial}|z|^2}
 $$
  whenever $|z|\ge b=b(\alpha,p)\gg 1$. Since $\mu\ge \lambda$, it follows from  $(\ref{eq:scalar})$ that
$$
  {\rm Tr}_g(\Theta_h) \le 1/{\lambda}(\rho^2)
$$
whenever $\rho\ge b\gg 1$.

For every holomorphic section $s$ of $L$ over $M$ we have
      \begin{eqnarray*}
      \Delta\left[\frac{p}4\log |s|^2_{h}+\frac{\varphi}2\right]
      &\ge & -p {\rm Tr}_g(\Theta_h)  + 2\Box \varphi\\
    &  \ge & \frac1{\lambda(\rho^2)}
            \end{eqnarray*}
            whenever  $\rho\ge b$. Set $\psi=\exp(\frac{p}4\log |s|^2_{h}+\frac{\varphi}2)$. Then we have
            \begin{equation}\label{eq:Delta}
  \Delta \psi \ge \frac{\psi}{\lambda(\rho^2) }\ \ \ \text{for \ \ } \rho\ge b.
  \end{equation}
Let $\chi$ and $\sigma_\delta$ be chosen as in the proof of Theorem \ref{th:finite_0}.  We set
  $$
  \phi=[1-\chi(\rho/b)]\chi(R^{-1}\log \log \rho),\ \ \ R\gg b,
  $$
and $\psi_\delta:=\delta\sigma(\psi/\delta)$.   Applying $(\ref{eq:Laplace+})$ with $\eta(t)=t$ and $\gamma=1/2$, we have
  $$
  \int_M \phi^2\psi_\delta \Delta \psi_\delta  dV \le 2 \int_M \psi_\delta^2 |\nabla \phi|^2 dV
    $$
    so that
    $$
    \int_{\{\phi=1\}\cap \{\psi\ge \delta\}}\frac{\psi^2}{\lambda(\rho^2) } dV \le 2 \int_M \psi_\delta^2 |\nabla \phi|^2 dV
        $$
        in view of $(\ref{eq:Delta})$.
 Letting  $\delta\rightarrow 0$, we obtain
  \begin{eqnarray*}
  \int_{2b<\rho< \exp\exp R} \frac{\psi^2}{\lambda(\rho^2) } dV &\le &
   2\int_{b<\rho<2b} \frac{\psi^2}{b^2} dV+ \frac2 {R^2}\int_{R<\log\log \rho<2R} \frac{\psi^2}{ \rho^2 \log^2\rho} dV.
   \end{eqnarray*}
   Letting $R\rightarrow +\infty$, we have
   $$
   \int_{\rho>2b} \frac{\psi^2}{\lambda(\rho^2)} dV \le \frac{2}{b^2}\int_{b<\rho<2b} {\psi^2} dV
      $$
      provided
      \begin{equation}\label{eq:LpCondition}
      \int_M \frac{\psi^2}{\lambda(\rho^2)} dV= \int_M |s|^p_{h} \frac{e^\varphi}{\lambda(\rho^2)} dV<\infty,
                  \end{equation}
                  for $(\rho\log \rho)^{-2}\le {\rm const}/\lambda(\rho^2) $ whenever $\rho\ge 1$.
      Thus
  $$
  \int_M \frac{\psi^2}{\lambda(\rho^2)} dV \le \int_{{\rho<2 b}} \frac{\psi^2}{\lambda(\rho^2)} dV+\frac2{b^2} \int_{b<\rho<2b} {\psi^2} dV,
    $$
    i.e.
    \begin{equation}\label{eq:KeyIneq}
    \int_M |s|^p_{h}  \frac{e^\varphi}{\lambda(\rho^2)} dV \le \int_{\rho<2b} |s|^p_{h}  \frac{e^\varphi}{\lambda(\rho^2)} dV+\frac2{b^2} \int_{b<\rho<2b}|s|^p_{h} e^\varphi  dV.
        \end{equation}

        Let $\widetilde{H}$ be the vector space of holomorphic sections of $L$ over $M$ which satisfy $(\ref{eq:LpCondition})$. We may verify similarly as the proof of Theorem \ref{th:finite_0} that  $\widetilde{H}$ is finite-dimensional.
 Since
 $$
 \int^\infty_1 \frac{dt}{\lambda(t)}\ge \int^t_{t/2} \frac{dt}{\lambda(t)}\ge \frac{t}{2\lambda(t)},
 $$
 we have $\lambda(t)\ge {\rm const}\cdot t$, so that
        $$
        \frac{e^\varphi}{\lambda(\rho^2)} =O(\rho^{-\alpha})
                        $$
            for $\rho$ large enough. Thus we have
         $\widetilde{ H}\supset \{s\in H^0(M,L):\int_M|s|_h^p\,\rho^{-\alpha}dV_g<\infty\}$, so that the latter is also finite-dimensional.
  \end{proof}
\section{Proof of Theorem \ref{cor:finite_cohomology}}
Let $(M,g)$ be an $n$-dimensional closed complex submanifold in $\C^N$, where $g$ is the restriction of the Euclidean metric, and let $L$ be a holomorphic line bundle over $M$ with a Hermitian metric $h$. Let $\{U_\alpha\}$ be an open covering of $M$ such that $L|_{U_\alpha}$ is trivial and $h$ may be written as $h=\{h_\alpha\}=\{e^{-\varphi_\alpha}\}$, where $\varphi_\alpha$ is a smooth function on $U_\alpha$. We 
also write $h=e^{-\varphi}$, where $\varphi=\{\varphi_\alpha\}$. Under these notations we have
\[
\Theta_h|_{U_\alpha}=i\partial\overline{\partial}\varphi|_{U_\alpha}:=i\partial\overline{\partial}\varphi_\alpha, \ \ \ \forall\,\alpha.
\]
 Note that $i\partial\overline{\partial}\varphi$ is a  globally defined (1,1)-form  on $M$ since $L$ is a holomorphic line bundle.

 Let $\mathcal{D}_{p,q}(M,L)$ be the set of  smooth $L$-valued $(p,q)$-forms on $M$ with a compact support and let $L^2_{p,q}(M,L)$ be the completion of  $\mathcal{D}_{p,q}(M,L)$ with respect to the  norm $\|\cdot\|_h^2:=\int_M|\cdot|_{g,h}^2dV_g$. Let  $\overline{\partial}^*_\varphi$ denote  the formal adjoint of $\overline{\partial}:\mathcal{D}_{p,q}(M,L)\rightarrow{\mathcal{D}_{p,q+1}(M,L)}$. The operators  $\overline{\partial}$ and $\overline{\partial}^*_\varphi$  extend to (unbounded) operators on $L^2_{p,q}(M,L)$, which are still denoted by the same symbols. The Hilbert space of $L$-valued harmonic $(p,q)$-forms on $M$ is defined to be
\[
\mathcal{H}^{p,q}_{L^2}(M,L):=\left\{u\in{L^2_{p,q}(M,L)}: \overline{\partial}u=\overline{\partial}^*_\varphi{u}=0\right\}.
\]
The Hodge isomorphism theorem asserts that
\[
\mathcal{H}^{p,q}_{L^2}(M,L)\cong \left.(\mathrm{Ker}\,\overline{\partial}\cap{L^2_{p,q}(M,L)})\right/\overline{\mathrm{Im}\,\overline{\partial}\cap{L^2_{p,q}(M,L)}},
\]
where the RHS is the reduced $L^2$ $\overline{\partial}$-cohomology group.   On the other hand, 
 the Bochner-Kodaira-Nakano identity (see e.g., \cite{DemaillyBook}) gives
\begin{equation}\label{eq:BKN}
\lVert{\overline{\partial}u}\rVert_h^2+\lVert{\overline{\partial}^*_\varphi{u}}\rVert_h^2\geq\int_M\big\langle{[i\partial\overline{\partial}\varphi,\Lambda]u,u}\big\rangle_hdV,
\end{equation}
where $\Lambda$ is the formal adjoint of the Lefschetz operator $L=\omega\wedge \cdot$  with $\omega$ being the K\"ahler form of $g$. For any $x\in{M}$ one can choose a local coordinate $(\zeta_1,\cdots,\zeta_n)$ on $M$ such that $\zeta_j(x)=0$ for $1\leq{j}\leq{n}$ and
\[\omega=\frac{i}{2}\partial\overline{\partial}|\zeta|^2+O(|\zeta|),\ \ \ i\partial\overline{\partial}\varphi=\frac{i}{2}\sum^n_{j=1}\lambda_jd\zeta_j\wedge{d\overline{\zeta}_j}+O(|\zeta|)\]
at $x$, where $\lambda_1,\cdots,\lambda_n$ are the eigenvalues of $i\partial\overline{\partial}\varphi$ with respect to $\omega$ and they are independent of the choice of  coordinates. We set
\[\lambda_{IJ}(\varphi)=\sum_{j\in{I}}\lambda_j+\sum_{j\in{J}}\lambda_j-\sum^n_{j=1}\lambda_j.\]
For $u=\sum'_{I,J}u_{IJ}d\zeta_I\wedge{d\overline{\zeta}_J}\otimes\xi$, where $\xi$ is a local frame of $L$, we have
\begin{equation}\label{eq:BKN RHS}
\big\langle{[i\partial\overline{\partial}\varphi,\Lambda]u,u}\big\rangle_h={\sum}'_{I,J}\lambda_{IJ}(\varphi)|u_{IJ}|^2e^{-\varphi}.
\end{equation}
Here $\sum'$ means that the summations are taken over increasing multi-indices.

\begin{theorem}\label{th:finite_cohomology}
 Suppose there exists a positive number $r$ such that
$$
{\sum}'_{I,J}\lambda_{IJ}(\varphi)\rho^2+2(n-p-q-2)\geq_{\mathrm{as}}r.
$$
Then $\dim\mathcal{H}^{p,q}_{L^2}(M,L)<\infty$.
\end{theorem}

To prove Theorem \ref{th:finite_cohomology}, we need the following lemma on eigenvalues of $i\partial\overline{\partial}\log|z|^2$ with respect to $i\partial\overline{\partial}|z|^2$, where $z=(z_1,\cdots,z_N)$ is the standard coordinate on $\C^N$.

\begin{lemma}\label{lm:eigenvalues}
For any $z_0\in\mathbb{C}^N$ there exists a local coordinate $w=(w_1,\cdots,w_N)$ near $z_0$ such that $w(z_0)=0$ and
\[
\omega=\frac{i}{2}\partial\overline{\partial}|z|^2=\frac{i}{2}\sum^N_{k=1}dw_k\wedge{d\overline{w}_k},\ \ \ i\partial\overline{\partial}\log|z|^2=\frac{i}{2}\sum^N_{k=2}\frac{2}{|z|^2}dw_k\wedge{d\overline{w}_k}
\]
at $z_0$.
\end{lemma}

\begin{proof}
 First of all, we have
\[i\partial\overline{\partial}\log|z|^2=\frac{i\partial\overline{\partial}|z|^2}{|z|^2}-\frac{i\partial|z|^2\wedge\overline{\partial}|z|^2}{|z|^4}.\]
 Choose a local coordinate $w$ at $z_0$ such that $w(z_0)=0$,
\[\frac{\partial}{\partial{w_1}}=\sum^N_{j=1}\frac{z_j}{|z|}\frac{\partial}{\partial{z_j}}\]
and $\partial/\partial{w_1},\cdots,\partial/\partial{w_N}$ are orthonormal at $z_0$.   Write $\partial/\partial{w_k}=\sum^N_{j=1}a_{kj}\partial/\partial{z_j}$ at $z_0$ for some unitary matrix $(a_{kj})$, so that
$dz_j=\sum^N_{k=1}a_{kj}dw_k$ at $z_0$. Hence
\[i\partial\overline{\partial}|z|^2=i\sum^N_{j=1}dz_j\wedge{d\overline{z}_j}=i\sum^N_{k=1}dw_k\wedge{d\overline{w}_k}\]
at $z_0$. Moreover, we have
\begin{eqnarray*}
i\partial|z|^2\wedge\overline{\partial}|z|^2&=&\sum_{j,l}\overline{z}_jz_ldz_j\wedge{d\overline{z}_l}\\
&=&\sum_{j,l,k,m}\overline{z}_jz_la_{kj}\overline{a}_{ml}dw_k\wedge{d\overline{w}_m}
\end{eqnarray*}
at $z_0$. Since $a_{1j}=z_j/|z|$ at $z_0$, we have $\sum^N_{j=1}a_{kj}\overline{z}_j=|z|\sum^N_{j=1}a_{kj}\overline{a}_{1j}=|z|\delta_{1k}$, so that
\[
i\partial|z|^2\wedge\overline{\partial}|z|^2=\sum_{k,m}|z|^2\delta_{1k}\delta_{1m}dw_k\wedge{d\overline{w}_m}=|z|^2dw_1\wedge{d\overline{w}_1}.
\]
Hence
\[i\partial\overline{\partial}\log|z|^2=\frac{i}{|z|^2}\sum^N_{k=2}dw_k\wedge{d\overline{w}_k}.\]
\end{proof}
\begin{proof}[Proof of Theorem \ref{th:finite_cohomology}]
Let $\psi$ be a real-valued smooth function on $M$. Consider a new Hermitian metric $\widetilde{h}=e^{-\varphi-\psi}$ of $L$. Since $\overline{\partial}^*_{\varphi+\psi}=\overline{\partial}^*_\varphi+\overline{\partial}\psi\lrcorner$, where ``$\lrcorner$'' is the contarction operator, it follows from  the Bochner-Kodaira-Nakano inequality and the Cauchy-Schwarz inequality that
\begin{equation}\label{eq:BKN II}
\lVert{\overline{\partial}u}\rVert^2_{\widetilde{h}}+(1+\varepsilon^{-1})\lVert{\overline{\partial}^*_\varphi{u}}\rVert^2_{\widetilde{h}}
\geq \int_M\big\langle{[i\partial\overline{\partial}(\varphi+\psi),\Lambda]u,u}\big\rangle_{\widetilde{h}}dV -(1+\varepsilon)\lVert{\overline{\partial}\psi\lrcorner{u}}\rVert^2_{\widetilde{h}}
\end{equation}
for any $u\in\mathcal{D}_{p,q}(M,L)$ and $\varepsilon>0$.

Let $\rho$ be the restriction of the function $|z|$ to $M$ and take $\psi=-\log\rho^2$. Then $i\partial\overline{\partial}\psi$ is the restriction of $i\partial\overline{\partial}\log|z|^2$ to the tangent spaces of $M$. Let $\lambda_1,\cdots,\lambda_n$ be defined as above for $\psi$ at a given point $x$. We may assume   $\lambda_1\leq\cdots\leq\lambda_n$. Let $V$ be the subspace of $\mathbb{C}^N$ spanned by the vectors $\partial/\partial{w}_2,\cdots,\partial/\partial{w_N}$ at $x$ which are defined as in Lemma \ref{lm:eigenvalues}. Then $\dim_\C{T_xM\cap{V}}=n-1$ and Lemma \ref{lm:eigenvalues} implies that $\lambda_1=\cdots=\lambda_{n-1}=-2\rho^{-2}$ and $-2\rho^{-2}\leq\lambda_n\leq0$. Hence 
\[
\lambda_{IJ}(\psi)\geq2(n-1-p-q)\rho^{-2},
\]
so that
\begin{eqnarray*}
\big\langle{[i\partial\overline{\partial}\psi,\Lambda]u,u}\big\rangle_{\widetilde{h}}
& \geq &
2(n-1-p-q)\rho^{-2}|u|_{\widetilde{h}}^2\\
& = &
2(n-1-p-q)|u|_h^2.
\end{eqnarray*}
Since
\[
|\bar{\partial}\psi|^2=\frac{1}{2}|d\psi|^2=\frac{1}{2}|\nabla\psi|^2=\frac{2|\nabla\rho|^2}{\rho^2}\leq\frac{2}{\rho^2}
\]
By \eqref{eq:BKN II}, we have
\begin{eqnarray}\label{eq:BKN III}
\lVert{\overline{\partial}u}\rVert^2_{\widetilde{h}}+(1+\varepsilon^{-1})\lVert{\overline{\partial}^*_\varphi{u}}\rVert^2_{\widetilde{h}}
 & \ge &   \int_M\big\langle{[i\partial\overline{\partial}\varphi,\Lambda]u,u}\big\rangle_{\widetilde{h}}dV\nonumber\\
&& + 2(n-p-q-2-\varepsilon)\lVert{u}\rVert^2_h.
\end{eqnarray}

Fix $u\in\mathcal{H}^{p,q}_{L^2}(M,L)$. Suppose  $\lambda_{IJ}(\varphi)\rho^2+2(n-p-q-2)\geq_{\mathrm{as}}{r}>0$ on $\{\rho>b\}$. For any $R>>b$, let $\phi$ be the cut-off function as in the proof of Theorem \ref{th:finite_0}, and choose $\varepsilon=r/4$. Since $\overline{\partial}u=\overline{\partial}^\ast_\varphi{u}=0$, it follows from \eqref{eq:BKN III}  that there exists a constant $C=C(r)>0$ such that
\begin{eqnarray*}
\lVert{\phi{u}}\rVert_h^2 & \leq & C\left(\lVert{\overline{\partial}(\phi{u})}\rVert_{\widetilde{h}}^2+\lVert{\overline{\partial}^*_\varphi(\phi{u})}\rVert_{\widetilde{h}}^2\right)\\
& \lesssim & \lVert{|\nabla\phi|\cdot{u}}\rVert_{\widetilde{h}}^2\\
& \lesssim & \int_{b<\rho <2b}|u|^2_hdV+C\int_{R<\rho <2R}|u|^2_hdV.
\end{eqnarray*}
Letting $R\rightarrow\infty$ we get
\[\int_{\rho>2b}|u|^2_hdV\lesssim\int_{b<\rho<2b}|u|^2_hdV\]
so that
\begin{equation}\label{eq:finiteness_cohomology}
\lVert{u}\rVert^2_h\lesssim\int_{\rho\leq{2b}}|u|^2_hdV.
\end{equation}

Let $\{u_j\}$ be a sequence of elements in $\mathcal{H}^{p,q}_{L^2}(M,L)$ such that $\lVert{u_j}\rVert_h\leq{1}$.
G$\mathring{\mathrm{a}}$rding's  inequality combined with Sobolev's estimate yields that $\{u_j\}$ is a normal family so that we have a subsequence $\{u_{j_k}\}$ converging locally uniformly to some $L$-valued harmonic form $u$ on $M$. Moreover, Fatou's lemma implies that $\lVert{u}\rVert_h\leq1$, in particular, $u\in\mathcal{H}^{p,q}_{L^2}(M,L)$. By \eqref{eq:finiteness_cohomology} we have $\lVert{u_{j_k}-u}\rVert\rightarrow0$ when $k\rightarrow\infty$. Hence $\dim_\C\mathcal{H}^{p,q}_{(2)}(M,L)<\infty$ in view of Proposition \ref{prop:finitedimen}.
\end{proof}

Theorem \ref{th:finite_cohomology} combined with Serre's duality immediately gives Theorem \ref{cor:finite_cohomology}.

\section{Proof of Theorem \ref{th:affine}}

 Let us  write the coordinate in $\mathbb C^N$ by $z=(z',z'')$, where $z'=(z_1,\cdots,z_n)$ and $z''=(z_{n+1},\cdots,z_N)$. Then we have

\begin{proposition}[see \cite{Chirka}, Chapter 1, \S\,7.3, \S\,7.4]\label{prop:affine criteria}
Let $M$ be a smooth $n$-dimensional  affine-algebraic variety in $\mathbb C^N$. Then there exists, after some unitary transformation,  a constant $C>0$ such that $|z''|\leq{C}(1+|z'|)$ for all $z\in{M}$ and the projection
\[\pi:M\rightarrow\C^n,\ \ \ z\mapsto{z'}\]
is a branched covering.
\end{proposition}

Let $S$ denote the branched locus of $\pi$. $S$ is an analytic set in $\mathbb{C}^n$ and
\[
\pi:M\setminus\pi^{-1}(S)\rightarrow\C^n\setminus{S}
\]
is an unbranched covering. Let $k$ be the number of sheets of this covering. For each $z'\in\C^n\setminus{S}$, we denote by $f_1(z'),\cdots,f_k(z')$ the  pre-images of $z'$. Locally, we may choose $f_j$ to be holomorphic functions. Set $dz':=dz_1\wedge\cdots\wedge{dz_n}$. Clearly, $(dz')^{\otimes m}$ may be used as a local holomorphic  frame of $K_M^{\otimes m}$ on $M\setminus\pi^{-1}(S)$.

\begin{proof}[Proof of Theorem \ref{th:affine}]
Fix $s\in{H^0_{L^p}(M,K_M^{\otimes m})}$. For each $z'_0\in\mathbb C^n\setminus{S}$ and $1\le j\le k$, there exists a neighborhood of $(z'_0,f_j(z'_0))$ where $s$ can be represented as
\[
s(z)={s}_j^\ast (z',f_j(z'))(dz')^{\otimes m}.
\]
Let $dV_M$ and $dV_{z'}$ denote the volume element of $M$ and $\C^n$ respectively.
We may write $dV_M=\lambda_j(z')dV_{z'}$ at $(z',f_j(z'))$.   If $h_m:=(dV_M)^{-m}$, then
\[
|s|^2_{h_m}=|{s}^\ast_j|^2\lambda_j^{-m} \ \ \ \text{at}\ (z',f_j(z')),
\]
so that
\begin{equation}\label{eq:translation}
\int_M|s|^p_{h_m}dV_M=\int_{\C^n\setminus{S}}\sum^k_{j=1}\left(|{s}^\ast_j|^2\lambda_j^{-m+{2}/{p}}\right)^{{p}/{2}}dV_{z'}.
\end{equation}
Although  ${s}^\ast_j$ and $\lambda_j$ are only locally defined, the sum in the RHS of $(\ref{eq:translation})$ is globally defined on $\C^n\setminus{S}$.

To proceed the proof, we need a calculation of $\lambda_j$ due to Demailly \cite{DemaillyAffine}. Suppose $M$ is defined by complex polynomials  $P_1,\cdots,P_r\in\mathbb C[z_1,\cdots,z_N]$. Set $\nu=N-n$ (in general, we have $\nu\leq{r}$). For each pair of multiindices
\begin{eqnarray*}
K&=&\{k_1<\cdots<k_\nu\}\subset\{1,\cdots,r\},\\
L&=&\{l_1<\cdots<l_\nu\}\subset\{1,\cdots,N\},
\end{eqnarray*}
we define the partial Jacobian by
\[
J_{KL}:=\det\left({\partial{P_{k_i}}}/{\partial{z_{l_j}}}\right)_{1\leq{i,j}\leq\nu}.
\]
For each $K$, let $U_K\subset{M}$ be the open set consisting of all points at which $dP_{k_1},\cdots,dP_{k_\nu}$ are linearly independent. Notice that $M\setminus{U_K}=\bigcap_L  J_{KL}^{-1}(0)$  is an analytic subset of $M$. Since the projection $\pi:M\rightarrow\C^n$  is a proper holomorphic map and $\pi^{-1}(S)\cup(M\setminus{U_K})$ is an analytic set in $M$, it follows from a  well-known theorem of Remmert (see, e.g., \cite{Chirka}, 5.8) that $\pi(U_K\setminus\pi^{-1}(S))=\mathbb C^n\setminus{S_K}$, where $S_K$ is an analytic set in $\mathbb C^n$.
 For the sake of simplicity, we  denote $J_{KL}(z',f_j(z'))$  by $J_{KL,j}(z')$ for $z'\in \mathbb C^n\setminus{S_K}$. It was proved by Demailly that
 \[
 \lambda_j (z')^{-1}=2^{-n} |J_{KL_0,j}(z')|^2\big/\sum_{L}|J_{KL,j}(z')|^2,\ \ \ z'\in \mathbb C^n\setminus{S_K}
 \]
where $L_0=(n+1,\cdots,N)$ (cf. \cite{DemaillyAffine}, \S\,10).
Set $\widetilde{S}=\bigcup_K{S_K}$. Then we have
 \begin{equation}\label{eq:Affine_Key}
 \lambda_j (z')^{-1}=2^{-n}\sum_K|J_{KL_0,j}(z')|^2\big/\sum_{K,L}|J_{KL,j}(z')|^2,\ \ \ z'\in \mathbb C^n\setminus{\widetilde{S}}.
 \end{equation}
 Since $M$ is nonsingular, we see that the term $\sum_{K,L}(\cdots)$ is nowhere vanishing.

Notice that
\begin{eqnarray}\label{eq:Affine_3}
 && \int_{\mathbb{C}^n\setminus\widetilde{S}}\sum^k_{j=1}\left(|{s}^\ast_j|^2 \frac{(\sum_K|J_{KL_0,j}|^2)^m}{(\sum_{K,L}|J_{KL,j}|^2)^m}\right)^{{p}/{2}}dV_{z'}\nonumber\\
&\leq&\int_{\mathbb{C}^n\setminus\widetilde{S}}\sum^k_{j=1}\left(|{s}^\ast_j|^2\frac{(\sum_K|J_{KL_0,j}|^2)^{m-{2}/{p}}}{(\sum_{K,L}|J_{KL,j}|^2)^{m-{2}/{p}}}\right)^{{p}/{2}}dV_{z'}\nonumber\\
&\le& {\rm const}_{m,n,p} \cdot \int_M|s|^p_{h_m}dV_M
\end{eqnarray}
in view of $(\ref{eq:translation})$ and $(\ref{eq:Affine_Key})$, and for $z'\in {\mathbb C^n\setminus\widetilde{S}}$,
\begin{eqnarray}\label{eq:Affine_4}
&& \sum^k_{j=1}\left(|{s}^\ast_j(z')|^2\frac{(\sum_K|J_{KL_0,j}(z')|^2)^m}{(\sum_{K,L}|J_{KL,j}(z')|^2)^m}\right)^{{p}/{2}}\nonumber\\
&=&2^{mnp/2}\sum^k_{j=1}\left(|{s}^\ast_j(z')|^2\lambda_j(z')^{-m}\right)^{{p}/{2}}\nonumber\\
&=&2^{mnp/2}\sum^k_{j=1}|s(z',f_j(z'))|_{h_m}^p.
\end{eqnarray}
Set
\[
\psi (z'):=\sum^k_{j=1}\left(|{s}^\ast_j(z')|^2\left({\sum}_K|J_{KL_0,j}(z')|^2\right)^m\right)^{{p}/{2}},\ \ \ z'\in\C^n\setminus\widetilde{S}.
\]
Since  ${s}^\ast_j$ and $f_j$ are locally holomorphic  and $\psi$ is globally defined on $\C^n\setminus\widetilde{S}$, it follows that $\psi$ is  psh on $\C^n\setminus\widetilde{S}$. Since $(\sum_{K,L,j}|J_{KL,j}(z')|^2)^m$ is locally bounded in $\C^n$, we infer from
$(\ref{eq:Affine_4})$ that  $\psi$ is bounded above near every point in  $\widetilde{S}$, thus extends  a psh function on $\C^n$, which is  denoted by the same symbol. It is also easy to see that $\log \psi$ is also psh on $\mathbb C^n$.

By $(\ref{eq:Affine_3})$ we have
\[
\int_{\C^n}\frac{\psi}{(\sum_{K,L,j}|J_{KL,j}|^2)^{mp/2}}dV_{z'}\lesssim\int_M|s|_{h_m}^pdV_M.
\]
Since $\sum_{K,L}|J_{KL}|^2$ is a sum of real polynomials and $|z''|\leq{C}(1+|z'|)$ on $M$, there exists an integer $\beta>0$ such that
\[
\frac{1}{(\sum_{K,L}|J_{KL}|^2)^{mp/2}}\gtrsim\frac{1}{(1+|z'|^2)^\beta}.
\]
Thus
\[
\int_{\C^n}\frac{\psi}{(1+|z'|^2)^\beta}\lesssim\int_{M}|s|^p_{h_m}dV_M<\infty.
\]
Since $\psi$ is a psh function, it follows from the mean-value inequality that for any $z'\in \C$,
\begin{eqnarray*}
  \psi (z') & \le & {\rm const}_n\cdot \int_{\{|\zeta'-z'|<1\}} \psi dV_{\zeta'}\\
 &\le&  {\rm const}_n\cdot \sup_{\{|\zeta'-z'|<1\}}(1+|\zeta'|^2)^\beta\int_{\{|\zeta'-z'|<1\}}\frac{\psi}{(1+|\zeta'|^2)^\beta}dV_{\zeta'}\\
& \lesssim&\lVert{s}\rVert^p_{L^p,h_m}(1+|z'|^2)^\beta,
\end{eqnarray*}
i.e.,  $\psi$ is of polynomial growth.
    Since $\log{\psi}$ is psh and $\psi$ is smooth outside some analytic set in $\C^n$, it follows from  Lemma \ref{lm:psh functions with polynomial growth} below that  $\mathrm{ord}_{a'}(\psi)\le {2\beta+4n+4}$ for some $a'\in\C^n\setminus\widetilde{S}$.
   Let $a'_1,\cdots,a'_k$ be the pre-images of $a'$ with respect to the projection $\pi$. From the definition of $\psi$ we know that
\[
\min_{1\leq{j}\leq{k}}\mathrm{ord}_{a'_j}(s)\leq(2\beta+4n+4)/p.
\]
Since $\beta$ is independent of $s$, it follows immediately from  Lemma \ref{lm:finite} below that  $P_{m,L^p}(M)<\infty$.
\end{proof}

\begin{lemma}\label{lm:psh functions with polynomial growth}
Let $\psi$ be a nonnegative psh function on $\C^n$ such that  $\log{\psi}$ is also  psh. Suppose that $\psi$ is positive and smooth outside an analytic set in $\mathbb{C}^n$ and
\[
\psi(z)\leq{C}(1+|z|^2)^\beta,\ \ \ z\in\C^n
\]
for some positive constants $C,\,\beta$. Then
$$
\mathrm{ord}_a(\psi)\le 2\beta+4n+4
$$
for all $a\in\C^n$ at which $\psi$ is positive and smooth.
\end{lemma}

As usual,  the order of a smooth function $f$ at some point $a$ is defined to be
\[
\mathrm{ord}_a(f):=\max\{k: \partial^\alpha{f}(a)=0,\, \forall\,|\alpha|<k\}.
\]
(If the set in RHS is empty, then we define $\mathrm{ord}_a(f)$  to be zero).

\begin{proof}[Proof of Lemma \ref{lm:psh functions with polynomial growth}]
Fix a smooth point $a$ of $\psi$. Without loss of generality, we may assume that   $A:=\mathrm{ord}_a(\psi)>2n$. Choose another point $b\in\C^n$ with $|b|<1$ and $\psi(b)>0$. Set
\[
\varphi:=\log{\psi}+2n\log|z-b|.
\]
Clearly, $\varphi$ is also  psh on $\mathbb{C}^n$.
Let $\chi$ be a smooth function which is supported in a sufficiently small neighborhood of $b$, such that  $\chi=1$ near $b$ and  $\chi=0$ near $a$.  Thanks to H\"{o}rmander's $L^2$-estimates for $\bar{\partial}-$operator  (see e.g., \cite{DemaillyBook}), there exists a solution of  $\bar{\partial} u =\bar{\partial}\chi$ such that
\[
\int_{\C^n}\frac{|u|^2}{\psi |z-b|^{2n}(1+|z|^2)^2}=\int_{\C^n}\frac{|u|^2 }{(1+|z|^2)^2} e^{-\varphi}\le \frac{1}{2}\int_{\C^n}|\bar{\partial}\chi|^2e^{-\varphi}<\infty.
\]
Since $\mathrm{ord}_a(\psi)=A>2n$ and $u$ is holomorphic near $a$, it follows that $\mathrm{ord}_a(u)\geq{A-2n}$, in particular,  $u(a)=0$. Moreover, we have $u(b)=0$. Thus
$f:=\chi-u$
gives a nonconstant entire function in $\C^n$. As $|z-b|^{2n}\lesssim(1+|z|^2)^n$ and $\psi(z)\lesssim (1+|z|^2)^\beta$, we conclude that
\[
\int_{\C^n}\frac{|f|^2}{(1+|z|^2)^{\beta+n+2}}<\infty.
\]
Again the mean-value inequality yields
\[
|f(z)|^2 \lesssim (1+|z|^2)^{\beta+n+2}.
\]
It follows from Cauchy's estimates that $f^{(\alpha)}(a)=0$ whenever $|\alpha|>2(\beta+n+2)$. Thus
\[
A \le 2n+\mathrm{ord}_a(u) = 2n+\mathrm{ord}_a(f)\le  2\beta+4n+4.
\]
\end{proof}

\begin{lemma}\label{lm:finite}
Let $M$ be an $n$-dimensional complex manifold and  $L$  a holomorphic line bundle over $M$. Let $a_1,\cdots,a_k$ be distinct points in $M$. Let $E$ be a complex subspace of $H^0(M,L)$. If there exists some constant $A>0$ such that
\[
\min_{1\leq{j}\leq{k}}\,\mathrm{ord}_{a_j}(s)\le {A}
\]
for all $s\in{E\setminus\{0\}}$, then $E$ is finite-dimensional. Here  $\mathrm{ord}_{a_j}(s)$ is defined to be $\mathrm{ord}_{a_j}({s}^\ast)$ where $s^\ast$ is a local representation of $s$.
\end{lemma}

\begin{proof}[Proof of Lemma \ref{lm:finite}]
For any nonzero sections $s_1,\cdots,s_l$  in $E$, we consider the following system of linear equations:
\[
c_1\partial^\alpha{s_1(a_j)}+\cdots+c_l\partial^\alpha{s_l(a_j)}=0,\ \ \ |\alpha|<{A+1},\ 1\leq{j}\leq{k}
\]
with indeterminates $c_1,\cdots,c_l$. Here the derivative $\partial^\alpha{s_m(a_j)}$ is defined to be the derivative of the  local representation  of $s_m$ with respect to some holomorphic frame near $a_j$. We denote by $l_0$ the number of equations. If  $l>l_0$, then we have nontrivial solutions $c_1,\cdots,c_l$. It follows that  $s:=\sum^l_{m=1}c_ms_m\in{E}$ and
\[
\mathrm{ord}_{a_j}(s)>{A},\ \ \ 1\leq{j}\leq{k},
\]
which implies $s=0$, i.e.,  $s_1,\cdots,s_l$ are linearly dependent. Thus  $\dim_\mathbb C E\le l_0$.
\end{proof}
\begin{remark}
{\rm Let $M$ be an $n$-dimensional affine-algebraic manifold in $\mathbb{C}^N$ defined by polynomials $P_1,\cdots,P_r$. Denote by $d:=\max_{1\le j\le r}\deg{P_j}$. The above argument also implies that  $P_{m,L^p}$ is bounded above by a constant  depending on $N,n,r,m$ and $p$. It would be interesting to compute the constant explicitly.}
\end{remark}

   \section{Infinite-dimensionality of $L^2$ plurigenera}
 A classical result of Bernstein states that if  $f:{\mathbb R}^2\rightarrow {\mathbb R}$ is a $C^2$ function then the graph $\Gamma(f)$ is  of $f$ in ${\mathbb R}^3$  is a minimal surface if and only if it is a plane. We shall prove an analogous result as follows. 

 \begin{proposition}\label{prop:vanishing}
 Let $M$ be a closed $1-$dimensional complex submanifold in $\mathbb C^N$. Then $
M$
 is a complex line if and only if $P_{m,L^2}(M)=0$ for all $m\ge 1.
  $
 \end{proposition}

 \begin{proof}
 Suppose that $M$ is a complex line. Without loss of generality, we may assume that $M$ is defined by $z_2=\cdots=z_N=0$. Since
 $g= dz_1\otimes d \bar{z}_1$, we see that
 \begin{eqnarray}\label{eq:equiv1}
  && f^\ast (dz_1)^{\otimes m} \in H^0_{L^2}(M,K_{M}^{\otimes m})\nonumber\\
  & \iff & f^\ast\in {\mathcal O}({\mathbb C})\ \ \text{and\ \ }\int_{\mathbb C} |f^\ast|^2 <\infty.
 \end{eqnarray}
Thus $f^\ast=0$, i.e., $P_{m,L^2}(M)=0$.

 Conversely, we shall first show that the Gauss curvature of $M$ vanishes through a standard application of  $L^2-$estimates for $\bar{\partial}$ on complete K\"ahler manifolds. Suppose on the contrary that ${\rm Gauss}<0$ near some point $x_0\in M$ (notice that we always have ${\rm Gauss}\le 0$). We choose a coordinate disc $(U_0,\zeta)$ around $x_0$ such that
 $
 {\rm Gauss}|_{U_0} \le -c<0,
 $
 and a cut-off function $\chi\in C_0^\infty(U_0)$ with $\chi=1$ in a small neighborhood of $x_0$. Let $m_0$ be a positive integer satisfying
 $$
 - (m_0-1)\cdot {\rm Gauss}+ i\partial\bar{\partial} (\chi\log |\zeta|^2)\ge c  i\partial\bar{\partial}|\zeta|^2
 $$
 on $U_0$.  It follows that there exists a solution $u$ of $\bar{\partial}u=\bar{\partial} \chi \wedge (d\zeta)^{\otimes m_0}$  such that
 $$
 \int_M |u|^2_{(dV)^{-m_0}} e^{-2 \chi \log |\zeta|^2} dV\le  \int_M \left(|\bar{\partial}\chi|^2/c\right) |(d\zeta)^{\otimes m_0}|^2_{(dV)^{-m_0}} e^{-2 \chi \log |\zeta|^2} dV<0.
  $$
  Thus $s:= \chi  (d\zeta)^{\otimes m_0}-u$ gives a nonzero section in $H^2_{L^2}(M,K_M^{\otimes m_0})$,
  which is a contradiction.

  By Proposition 1.2.13 in \cite{Fujimoto}, we know that $M$ has to lie in a plane. Since $M$ is a closed complex submanifold in $\mathbb C^N$, it follows that $M$ is a complex line.
 \end{proof}

 \begin{remark}
 \rm Analogously, we may show that if ${\rm dim}_{\mathbb C}\,M>1$ and  $P_{m,L^2}(M)=0$ for all $m$  then the Ricci form has at least one zero eigenvalue everywhere. It would be interesting to know whether the converse also holds.
 \end{remark}

 \begin{proof}[Proof of Theorem \ref{th:plurigenera_curve}]
 Since
 $g=(1+|f'|^2) dz\otimes d \bar{z}$, we see that
 \begin{eqnarray}\label{eq:equiv2}
  && f^\ast (dz)^{\otimes m} \in H^0_{L^2}(\Gamma(f),mK_{\Gamma(f)})\nonumber\\
  & \iff & f^\ast\in {\mathcal O}({\mathbb C})\ \ \text{and}\ \ \int_{\mathbb C} |f^\ast|^2 (1+|f'|^2)^{-m+1}<\infty.
 \end{eqnarray}
 The only if part follows from Theorem \ref{th:affine}. For the if part, it suffices to
 show $P_{3,L^2}(\Gamma(f))=\infty$, i.e.,
\begin{equation}\label{eq:equiv3}
\mathrm{dim}_{\mathbb C} \left\{f^\ast\in\mathcal{O}(\mathbb{C}): \int_{\mathbb{C}}|f^\ast|^2(1+|f'|^2)^{-2}<\infty\right\}=\infty
\end{equation}
whenever $f$ is transcendental.
Suppose at first that $f_1:=f'$ has an infinite number of  zeros $a_1,a_2,\cdots$. For
  $$
\varphi:=\log |f_1|^2 + \log(1+|f_1|^2),
$$
one has
$$
\varphi_{z\bar{z}}\ge \frac{|f'_1|^2}{(1+|f_1|^2)^2},
$$
so that for any $k$ there exists a disc $U_k$ around $a_k$ such that $a_j\notin U_k$, $\forall\,j\neq k$, and
$$
\varphi_{z\bar{z}}|_{U_k\backslash \frac12 U_k}\ge c_k>0.
$$
Here $\frac12 U_k$ means the disc with same center as $U_k$ but half of radius.
 Choose $\chi_k\in C^\infty_0(U_k)$ with $\chi_k|_{\frac12 U_k}=1$. There is a solution $u_k$ of $
\bar{\partial} u= \bar{\partial}\chi_k
$
such that
$$
\int_{\mathbb C} |u_k|^2 e^{-\varphi}\le \int_{\mathbb C}  |\bar{\partial}\chi_k|^2/c_k \, e^{-\varphi}<\infty.
$$
Thus $f^\ast_k:=\chi_k-u_k\in {\mathcal O}({\mathbb C})$ satisfies $f_k^\ast(a_k)=1$, $f^\ast_k(a_j)=0$, $\forall\,j\neq k$, and
 \begin{eqnarray*}
\int_{\mathbb C} \frac{|f^\ast_k|^2}{(1+|f_1|^2)^2} & \le & 2 \int_{\mathbb C} \frac{|\chi_k|^2}{(1+|f_1|^2)^2}+2\int_{\mathbb C}  \frac{|u_k|^2}{(1+|f_1|^2)^2}\\
& \le & 2 \int_{\mathbb C} \frac{|\chi_k|^2}{(1+|f_1|^2)^2} +2 \int_{\mathbb C} |u_k|^2 e^{-\varphi}<\infty.
\end{eqnarray*}
Thus $P_{3,L^2}(\Gamma(f))=\infty$.

For the  general case, we infer from Picard's big theorem that  there exists $b$ with $|b|<1/2$ such that $f^{-1}_1(b)$ is an infinite set.
Applying the above argument with $f_1$ replaced by $f_1-b$, we see that
$$
\mathrm{dim}_{\C}\left\{f^\ast\in\mathcal{O}(\mathbb{C}): \int_{\mathbb{C}}|f^\ast|^2(1+|f_1-b|^2)^{-2}<\infty\right\}=\infty.
$$
 Since
\[
\frac{1}{1+|f_1-b|^2}\geq \frac{1}{2(1+|f_1|^2)},
\]
 we conclude the proof.
 \end{proof}

 \begin{problem}
Suppose $f$ is transcendental. Is it possible to conclude that  $P_{2,L^2}(\Gamma(f))=\infty$?
\end{problem}

Notice that $P_{1,L^2}(\Gamma(f))=0$ always holds.

\begin{remark}
{\rm By Picard's little theorem, we know that the image of $f_1:=f'$  omits at most one point in $\mathbb C$. Suppose $f_1(\mathbb C)$ omits precisely a point, say $b\in \mathbb C$. Since $\sqrt{f_1-b}$ is also a transcendental entire function, it follows from the  argument above that }
$$
\mathrm{dim}_{\C}\left\{f^\ast\in\mathcal{O}(\mathbb{C}): \int_{\mathbb{C}}|f^\ast|^2(1+|f_1-b|)^{-2}<\infty\right\}=\infty.
$$
{\rm Since}
\[
\frac{1}{1+|f_1-b|}\ge \frac1{1+|b|}\cdot \frac{1}{1+|f_1|},
\]
{\rm we conclude that $P_{2,L^2}(\Gamma(f))=\infty$.}
\end{remark}

In general we have the following infiniteness criterion. 

  \begin{proposition}\label{th:infinite}
    Suppose $L$ is semipositive and there exist $\alpha>2n$  and a sequence of mutually disjoint balls $\{M(x_k, r_k)\}$ such that
   \begin{equation}\label{eq:Ricci}
   \Theta_h \ge  \alpha\cdot \omega/r_k^2 \ \ \ on\ \ M(x_k, r_k),\ \ \forall\,k,
\end{equation}
where $\omega$ is the K\"ahler form of $g$. Then we have
 $$
 {\rm dim}_{\mathbb C} H^0_{L^2}(M,K_M\otimes L)=\infty.
 $$
 \end{proposition}

  \begin{corollary}\label{cor:infinite_0}
Suppose $L$ is semipositive and there exists $\alpha>8n$  such that
  \begin{equation}\label{eq:trace_2}
 \Theta_h \ge_{{\rm as}}  \alpha\cdot \omega/\rho^2.
 \end{equation}
  Then
 $
{\rm dim}_{\mathbb C}H^0_{L^2} (M,K_M\otimes L)=\infty.
$
\end{corollary}

\begin{proof}
Since $\alpha>8n$,  there exists $0<\tau<1$ such that $\tilde{\alpha}:=\alpha/(1+1/\tau)^2>2n$. Choose $\{x_k\}\subset M$ with $\rho(x_k)\rightarrow \infty$. Set $r_k:=\tau \rho(x_k)$. Then we have
$$
\Theta_h \ge \alpha\cdot \omega/\rho^2 \ge \frac{\alpha\cdot  \omega}{(1+\tau)^2 \rho(x_k)^2}=\tilde{\alpha}\cdot \omega/r_k^2
$$
on $M(x_k,r_k)$ for all $k\gg 1$. After taking a subsequence, Proposition \ref{th:infinite} applies.
\end{proof}

  \begin{corollary}\label{cor:infinite}
 Suppose $m>1$ and there exists $\alpha>\frac{8n}{m-1}$  such that the Ricci curvature satisfies
 $$
{\rm Ric}_g \le_{{\rm as}} -\alpha/\rho^2.
$$
 Then $P_{m,L^2}(M)=\infty$.
 \end{corollary}

\begin{remark}
{\rm (1) Notice that ${\rm Ric}_g\le 0$  holds for any complex submanifolds in $\mathbb C^N$. (2) Corollary \ref{cor:finite} implies that if $n>1$ and ${\rm Ric}_g \ge_{{\rm as}} -\alpha/\rho^2$ for
 some $\alpha<\frac{n-1}{m}$, then $P_{m,L^2}(M)<\infty$ It is unclear whether this bound is the best possible.}
 \end{remark}

    To prove   Proposition \ref{th:infinite}, we need  the following construction of singular weight functions which is essentially due to Berndtsson-Ortega \cite{BerndtssonCerda} (see also Wu-Wang \cite{WuWang}).

        \begin{lemma}\label{lm:cutoff}
        Let $\kappa$ be a function on\/ ${\mathbb R}$ such that $\kappa|_{[0,\infty)}=0$ and $\kappa(t)=1+t-e^t$ for $t<0$. For $w\in {\mathbb C}^N$ and $R>0$ we set
        $$
        \varphi_{w,R}(z)=\kappa(\log (|z-w|^2/R^2)).
        $$
        Then we have
        \begin{enumerate}
        \item $\varphi_{w,R}(z)\sim \log |z-w|^2$ as $z\rightarrow w$,
       \item $ i \partial\bar{\partial} \varphi_{w,R}=0$ for $|z-w|> R$,
        \item $
        i\partial\bar{\partial} \varphi_{w,R} \ge -\frac1{R^2}\cdot i\partial\bar{\partial} |z|^2
        $
       for $|z-w|<R$.
       \end{enumerate}
        \end{lemma}

        \begin{proof}
       For the sake of completeness we shall give a proof here. Clearly, it suffices to verify {\it (3)}. Set $\phi(z):=\log (|z-w|^2/R^2)$. Then we have
       \begin{eqnarray*}
       i \partial\bar{\partial} \varphi_{w,R}  & = & \kappa'(\phi)i \partial\bar{\partial} \phi + \kappa''(\phi) i \partial\phi\wedge \bar{\partial} \phi\\
       & \ge & - \frac{|z-w|^2}{R^2} i \partial\phi\wedge \bar{\partial} \phi\\
      & = &  - \frac{i}{R^2|z-w|^2}  \sum_j \overline{(z_j-w_j)}dz_j \wedge  \sum_j {(z_j-w_j)}d\bar{z}_j\\
       & \ge & - \frac{1}{R^2}  i \partial\bar{\partial} |z|^2
       \end{eqnarray*}
       in view of the Schwarz inequality.
        \end{proof}

        \begin{proof}[Proof of Proposition \ref{th:infinite}]
        Set
       $
       \psi_{k}=\varphi_{x_k, r_k}|_M.
       $
       By Lemma \ref{lm:ddbar} and Lemma \ref{lm:cutoff} we have
       \begin{equation}\label{eq:ddbarLower}
       i\partial\bar{\partial} \psi_{k} \ge -2\omega/ r_k^2\ \ \ \text{on\ \  } M(x_k, r_k),
              \end{equation}
    and $ i\partial\bar{\partial} \psi_{k} =0$ on $M\backslash \overline{M(x_k,r_k)}$.
 We introduce the following singular Hermitian metric on $L$:
  $$
  \tilde{h}:=h e^{-n\sum_k \psi_{k}}.
  $$
  It follows from $(\ref{eq:ddbarLower})$   that
  $$
  \Theta_{\tilde{h}} \ge \left(\alpha-2n\right)  \omega/r_k^2
  $$
  on $M(x, r_k)$ for every $k$ and $\Theta_{\tilde{h}}\ge 0$ on $M$.

  Now fix $k$ for a moment. Let us choose a coordinate neighborhood $(U,\zeta)$ around $x_k$  and
   a local frame $\xi$ of $L$ over $U$. We may assume that $U\subset M(x_k,r_k)$. Take a smooth function $\chi_k$ such that $\mathrm{supp}\,\chi_k\subset{U}$ and $\chi_k=1$ in a neighborhood of $x_k$. Again we have a solution $u:=u_k$ to the equation
   $$
   \bar{\partial} u= \bar{\partial}\chi_k \wedge d\zeta_1\wedge\cdots\wedge d\zeta_n\otimes \xi
      $$
      which satisfies
      $$
      \int_M |u_k|^2 _{(dV)^{-1}\otimes {h}}\, e^{-n\sum_j \psi_{j}} dV<\infty.
      $$
      It follows that $s_k:=\chi_k d\zeta_1\wedge\cdots\wedge d\zeta_n\otimes \xi-u_k$ is a global holomorphic section
   of $K_M\otimes L$ over $M$ such that $s_k(x_k)\neq 0$, $s_k(x_j) = 0$ for all $j\neq k$, and
   $
   \int_M |s_k|^2_{(dV)^{-1}\otimes {h}} dV<\infty.
      $
      Thus we have
\[{\rm dim\,} H^0_{L^2}(M,K_M\otimes L)=\infty.\qedhere\]
\end{proof}

               Below we give a two-dimensional example satisfying the condition of Proposition \ref{th:infinite}.

                   \begin{example}
         {\rm  Consider $M=\Gamma(f)\hookrightarrow \mathbb C^3$ where $f(z_1,z_2)=e^{iz_1^2}+e^{iz_2^2}-2i z_1z_2$. Then we have}
        $$
g  =  (1+|f_{z_1}|^2)dz_1d\bar{z}_1+(1+|f_{z_2}|^2)dz_2d\bar{z}_2\\
 +f_{z_1}\overline{f_{z_2}}dz_1d\bar{z}_2+f_{z_2}\overline{f_{z_1}}dz_2d\bar{z}_1.
$$
{\rm Set $\varphi=\log (1+|f_{z_1}|^2+|f_{z_2}|^2)$. Then the Ricci tensor is given by $-i\partial\bar{\partial}\varphi$. It is easy to verify that}
$$
i\partial\bar{\partial}\varphi\ge \frac{i\partial f_{z_1}\wedge \overline{\partial f_{z_1}}+ i\partial f_{z_2}\wedge \overline{\partial f_{z_2}}}{(1+|f_{z_1}|^2+|f_{z_2}|^2)^2},
$$
$$
f_{z_1}=2i z_1 e^{iz_1^2}-2i z_2,\ \ f_{z_2}=2i z_2 e^{iz_2^2}-2i z_1,
$$
$$
f_{z_1z_2}=f_{z_2z_1}=-2i,\ \ f_{z_jz_j}=(-4z_j^2+2i)e^{iz_j^2},
$$
{\rm and}
$$
e^{iz_j^2}=e^{-2x_j y_j}\left(\cos(x^2_j-y^2_j)+i\sin(x^2_j-y^2_j)\right),\ \ \ z_j=x_j+iy_j,
$$
{\rm for $j=1,2$. Set $r_k=1/k$ and  }
$$
p_k=\left(\sqrt{2k\pi},\sqrt{2k\pi},f (\sqrt{2k\pi},\sqrt{2k\pi})\right).
$$
{\rm  For all $k\gg 1$ the balls
 $
 M(p_k,r_k)
 $
 are mutually disjoint and satisfy}
 $$
e^{-2x_j y_j}\sim 1+O(k^{-1/2}),\ \ \cos(x^2_j-y^2_j)\sim 1+O(k^{-1/2})
$$
$$
 \sin(x^2_j-y^2_j)= O(k^{-1/2})
$$
{\rm on each $
 M(p_k,r_k)
 $, so that }
 $$
|f_{z_j}|\lesssim 1,\ \ |f_{z_jz_j}|\asymp k
$$
 {\rm Since
$
g \lesssim dz_1d\bar{z}_1+dz_2d\bar{z}_2
$
on $ M(p_k,r_k)$,
 it follows that}
$$
{\rm Ric}_g|_{ M(p_k,r_k)} \lesssim -k^2= - 1/r_k^2.
$$
{\rm By  Proposition \ref{th:infinite} we see that  $P_{m,L^2}(M)=\infty$ for $m\gg 1$.}
           \end{example}

\begin{proof}[Proof of Corollary \ref{cor:affine}]
Suppose on the contrary that
$
\limsup_{\rho\rightarrow \infty} \rho^2\cdot {\rm Ric}_g < 0.
$
Then there exists $\alpha>0$ such that
$
{\rm Ric}_g \le_{\rm as} -\alpha/\rho^2.
$
We choose $m\gg 1$ so that $\alpha>8n/(m-1)$. It follows from Corollary \ref{cor:infinite} that $P_{m,L^2}(M)=\infty$, which contradicts with Theorem \ref{th:affine}.
\end{proof}
         \section{Appendix: A monotonic property}
         The following result is motivated by Gromov's work \cite{Gromov}.

 \begin{proposition}\label{prop:increasing}
  Let $(L,h)$ be a Hermitian line bundle over $M$. Suppose  there exists $\alpha>0$ such that
 $
  {\rm Tr}_g(\Theta_h) \le  \alpha.
 $
Then for \/ $0<p_1<p_2\le \infty $ we have
$$
H^0_{L^{p_1}} (M,L)\subset H^0_{L^{p_2}} (M,L).
$$
 \end{proposition}

 \begin{proof}
  We first show that for $0<p<\infty$,
  $$
  H^0_{L^{p}} (M,L) \subset H^0_{L^{\infty}} (M,L).
  $$
  For every $s\in H^0_{L^{p}} (M,L)$ we have
  $$
 \frac{p}2 \Delta \log |s|^2_h \ge -2p {\rm Tr}_g(\Theta_h)\ge -2p \alpha.
    $$
    Set $\tilde{g}= g+ idtd\bar{t}$,  $t\in {\mathbb C}$. It turns out that
    $\psi:=\log |s|^p_h+ p\alpha |t|^2$ is subharmonic with respect to $\tilde{g}$ on $\widetilde{M}:=M\times {\mathbb C}$, so is $e^{\psi}$. Since $\widetilde{M}$ is a minimal submanifold in ${\mathbb C}^{N+1}={\mathbb R}^{2N+2}$, we infer from the  mean-value inequality (cf. \cite{ColdingMinicozzi}, Corollary 1.17) that for every $x\in M$,
    \begin{eqnarray*}
    |s|^p_h(x) = e^{\psi(x,0)} & \le & \frac{\int_{\widetilde{M}((x,0),1)} e^{\psi(y,t)} dV_y dV_t }{{\rm vol}_{\rm eucl}(B_1\subset {\mathbb R}^{2n+2})}\\
   & \le & {\rm const}_n\cdot {\int_{M(x,1)\times {\mathbb D}} e^{\psi(y,t)} dV_y dV_t }\\
   & \le & {\rm const}_n\cdot \int_{M(x,1)} |s|^p_h dV \cdot \int_{\mathbb D} e^{p\alpha |t|^2} dV_t\\
  &  \le & {\rm const}_{n} \cdot \frac{e^{p\alpha}}{p\alpha} \cdot \int_{M} |s|^p_h dV,
          \end{eqnarray*}
          i.e., $\|s\|_{L^\infty}\le {\rm const}_{n,p,\alpha} \cdot \|s\|_{L^p}$.

         Now suppose $s\in H^0_{L^{p_1}} (M,L)$. We have
         $$
         \int_{M} |s|^{p_2}_h dV \le \int_{M} |s|^{p_1}_h dV \cdot \|s\|_{L^\infty}^{p_2-p_1} <\infty,
                         $$
                         i.e., $s\in H^0_{L^{p_2}} (M,L)$.
 \end{proof}

  \begin{corollary}\label{cor:increasing}
 Suppose  there exists $\alpha>0$ such that
 $$
 {\rm Scal}_g \ge  -\alpha.
 $$
Then $P_{m,L^p}(M)$ is non-decreasing in $p$.
 \end{corollary}

 \begin{remark}
 {\rm Fix $p<2$ and $k\in \mathbb Z^+$ with $kp\ge 2$. Suppose we have
 $$
 s_j\in H^0_{L^2}(M,K_M^{\otimes m_j})\subset H^0_{L^{kp}}(M, K_M^{\otimes m_j}),\ \ \  j=1,\cdots,k.
 $$
  Then}
 $$
 s:=s_1\otimes \cdots \otimes s_k\in H^0(M,K_M^{\otimes m}), \ \ \ m=m_1+\cdots+m_k.
  $$
 {\rm H\"older's inequality gives}
 $$
 \int_M |s|_{h^{\otimes m}}^{p} dV\le \prod_{j=1}^k \|s_j\|_{L^{kp}}^{p}<\infty,
 $$
 {\rm i.e., $s\in H^0_{L^{p}}(M,K_M^{\otimes m})$. Combined with Proposition \ref{prop:increasing} we may produce many $L^p$ pluri-canonical sections from $L^2$ pluri-canonical sections.
 }
 \end{remark}

  \subsection*{Acknowledgments} The first author would like to thank Prof. Takeo Ohsawa for bringing the reference \cite{Fujimoto} to his attention. The authors are also indebted to the referee for valuable comments.

\end{document}